\documentclass[11pt,reqno,xcolor=dvipsname]{amsart}
\usepackage{amssymb,mathrsfs,xcolor,bbold,comment,enumerate,graphicx,float}
\usepackage[colorlinks=true,urlcolor=OliveGreen,citecolor=purple,linkcolor=blue]{hyperref}
\numberwithin{equation}{section}
\usepackage{mathptmx}
\usepackage{tikz}
\usetikzlibrary{shapes,arrows,matrix,calc,positioning,shapes.geometric,shapes.symbols,shapes.misc}
\tikzstyle{mybox} = [draw=black, very thick, diamond, inner ysep=5pt, inner xsep=5pt]
\tikzstyle{myboxS} = [draw=black, thick, rectangle, rounded corners, inner ysep=5pt, inner xsep=5pt]

\tikzstyle{legenda} = [draw=black, thick, rectangle, inner ysep=5pt, inner xsep=5pt]

\usepackage[all]{xy}
\usepackage{color}
\usepackage{amsaddr}
\usepackage{chngcntr}
\usepackage{apptools}
\AtAppendix{\counterwithin{theorem}{section}}
\AtAppendix{\counterwithin{equation}{section}}
\usepackage[latin9]{inputenc}
\makeatletter
\renewcommand{\email}[2][]{
\ifx\emails\@empty\relax\else{\g@addto@macro\emails{,\space}}\fi
\@ifnotempty{#1}{\g@addto@macro\emails{\textrm{(#1)}\space}}
\g@addto@macro\emails{#2}
}
\makeatother

\makeatletter
\def\namedlabel#1#2{\begingroup
    #2%
    \def\@currentlabel{#2}%
    \phantomsection\label{#1}\endgroup
}
\makeatother

\makeatletter
\@namedef{subjclassname@2020}{%
  \textup{2020} Mathematics Subject Classification}
\makeatother

\theoremstyle{plain}
\newtheorem{theorem}{Theorem}[section]
\newtheorem{proposition}[theorem]{Proposition}
\newtheorem{corollary}[theorem]{Corollary}
\newtheorem{lemma}[theorem]{Lemma}

\theoremstyle{definition}
\newtheorem{assumption}[theorem]{Assumption}
\newtheorem{remark}[theorem]{Remark}
\newtheorem{example}[theorem]{Example}
\newtheorem{definition}[theorem]{Definition}

\newcommand{\la}{\langle}
\newcommand{\ra}{\rangle}
\newcommand{\mf}{\mathfrak}

\newcommand{\R}{\mathbb{R}}
\newcommand{\QW}{\mathbb{Q}}
\newcommand{\PW}{\mathbb{P}}
\newcommand{\Q}{\mathbb{Q}}

\newcommand{\E}{\mathbb{E}}
\newcommand{\tn}{\textnormal}

\newcommand{\ind}{\mathbf{1}}

\renewcommand{\P}{\mathbb{P}}
\newcommand{\Norm}{{\|\cdot\|}}
\newcommand{\N}{\mathbb{N}}

\newcommand{\Linfty}{{L^{\infty}_{\mathbf c}}}
\newcommand{\Linftyplus}{{L^{\infty}_{{\mathbf c}{\tiny +}}}}

\newcommand{\sca}{{sca_{\mathbf c}}}
\newcommand{\om}{\omega}

\newcommand{\CF}{\mathcal F}
\newcommand{\CA}{\mathcal A}
\newcommand{\ca}{{ca_{\mathbf c}}}
\newcommand{\CG}{\mathcal G}
\newcommand{\CE}{\mathfrak E}
\newcommand{\CX}{\mathcal X}
\newcommand{\CP}{\mathfrak P}
\newcommand{\CC}{\mathcal C}
\newcommand{\CY}{\mathcal Y}

\newcommand{\CL}{\mathcal L}
\newcommand{\CH}{\mathcal H}
\newcommand{\CB}{\mathcal B}
\newcommand{\CQ}{\mathfrak Q}
\newcommand{\CR}{\mathfrak R}

\newcommand{\CS}{\mathcal S}
\newcommand{\w}{\widehat}

\newcommand{\peq}{\preceq}

\newcommand{\mbf}{\mathbf}
\renewcommand{\la}{\langle}
\renewcommand{\ra}{\rangle}
\newcommand{\Lzeroplus}{L^0_{\mathbf{c}+}}
\newcommand{\mbb}{\mathbb}

\newcommand{\tle}{\trianglelefteq}
\newcommand{\Lzero}{{L^0_{\mbf c}}}

\title[Model Uncertainty: A reverse approach]{Model Uncertainty: A reverse approach}

\author[F.-B.~Liebrich]{Felix-Benedikt Liebrich}
\address{Institute of Actuarial and Financial Mathematics \& House of Insurance,\\
Leibniz Universit\"at Hannover, Germany}
\email{felix.liebrich@insurance.uni-hannover.de}

\author[M.~Maggis]{Marco Maggis}
\address{Department of Mathematics,\\
University of Milan, Italy}
\email{marco.maggis@unimi.it}

\author[G.~Svindland]{Gregor Svindland}
\address{Institute of Actuarial and Financial Mathematics \& House of Insurance,\\
Leibniz Universit\"at Hannover, Germany}
\email{\\gregor.svindland@insurance.uni-hannover.de}

\date{January 1, 2022}

\begin{document}

\parindent 0em \noindent

\begin{abstract} Robust models in mathematical finance replace the classical single probability measure by a sufficiently rich set of probability measures on the future states of the world to capture (Knightian) uncertainty about the ``right" probabilities of future events. If this set of measures is \textit{nondominated}, many results known from classical \textit{dominated} frameworks cease to hold as probabilistic and analytic tools crucial for the handling of dominated models fail. We investigate the consequences for the robust model when prominent results from the mathematical finance literature are postulate. In this vein, we categorise the Kreps-Yan property, robust variants of the Brannath-Schachermayer Bipolar Theorem, Fatou representations of risk measures,  and aggregation in robust models.

\smallskip


\smallskip

\end{abstract}

\maketitle

\section{Introduction}

This paper studies questions related to model uncertainty in nondominated frameworks. Model uncertainty or \textit{Knightian uncertainty} refers to situations in which the mechanism behind the realisation of economic or financial outcomes is ambiguous or impossible to be known.
Mathematically, these are usually modelled by a nonempty set $\CP$ of probability measures on a measurable space $(\Omega,\CF)$ of relevant future states which can be classified as follows: 
\begin{enumerate}[(a)]
\item \textsf{Pure risk:} $\CP=\{\P\}$ is a singleton. 
\item\textsf{Dominated uncertainty:} The set $\CP$ is dominated by a single reference probability measure $\P$ on $(\Omega,\CF)$, i.e.\ each $\P$-null event $N\in\CF$ satisfies $\sup_{\Q\in\CP}\Q(N)=0$.
\item\textsf{Nondominated uncertainty:} There is no dominating probability measure for $\CP$.
\end{enumerate}
Often models of nondominated uncertainty appear under the label {\em robust}; see \cite{BN, Burzoni, Nutz, STZ} and the references therein.
They are subject to crucial limitations though; as Bouchard \& Nutz \cite{BN} remark: 
\begin{quote}{\footnotesize The main difficulty in our endeavor is that [$\CP$] can be nondominated which leads to the failure of various tools of probability theory and functional analysis ... As a consequence, we have not been able to reach general results by using separation arguments in appropriate function spaces ...\quad\cite[p.\ 824]{BN}}\end{quote}
The mentioned probabilistic and analytic toolbox under dominated uncertainty tends to strongly rely on the continuity of a dominating probability measure $\PW$. Arguments based thereon facilitate many studies in financial mathematics. 
To give examples, law-invariant financial metrics such as Value-at-Risk fall in category (a) of pure risk, while financial applications that require a change of measure---as for instance in the context of the Fundamental Theorem of Asset Pricing (FTAP)---fall in category (b). In the latter case, the exhaustion principle is still applicable (see, e.g., the proof of the Halmos-Savage Theorem).  

The literature knows a number of prominent \textit{ad hoc} circumventions of the lack of a dominating measure in situation (c).
\begin{itemize}
\item Cohen's~\cite{Cohen} \textit{Hahn property}, a specific structure of the
underlying measure space in relation to the occurring probabilities; 
\item The assumption that the uncertainty structure is tree-like and that $\Omega$ is Polish, which in turn admits dynamic
programming and measurable selection arguments, see
\cite{Ba19,BZ,Car,BN,BN14}; 
\item Focusing on particularly well-behaved state spaces like the Wiener space, cf.\ \cite{NV,STZ11,STZ}.
\end{itemize}
This strand of literature imposes \textit{sufficient} conditions to guarantee that some nondominated robust model is a tractable setting for particular applications. In stark contrast, the present manuscript takes the opposite mathematical standpoint. Its \textit{reverse} perspective tackles the question which structural consequences for the triplet $(\Omega,\CF,\CP)$ can be derived from assuming that certain desired results from the mentioned probabilistic and analytic toolbox are available. In other words, the paper addresses the usual mathematical question of \textit{necessity} of sufficient conditions. 
We shall see that this leads to an additional subcategory (c') of nondominated uncertainty (c) which in this explicitness seems to be missing in the mathematical finance literature so far.
\begin{enumerate}
\item[(c')]\textsf{Supported uncertainty:} There is an alternative set $\CQ$ of probability measures on $(\Omega,\CF)$ which is equivalent to $\CP$ (an event $N\in\CF$ is {\em $\CP$-polar}, i.e.\ $\sup_{\P\in\CP}\P(N)=0$, if and only if it is $\CQ$-polar, i.e.\ $\sup_{\Q\in\CQ}\Q(N)=0$), and such that each measure $\Q\in\CQ$ is {\em supported}.
\end{enumerate}
The supports required in (c') are not to be understood in a statewise or topological, but in an order sense (cf.\ Definition~\ref{def:support}), and they are only unique up to $\CP$-polar events. We shall say that a set $\CP$ as described in (c') is \textit{of class }(S), while the alternative set of supported probability measures $\CQ$ is a \textit{supported alternative} to $\CP$. 

Results that are well known in situations (a)--(b) above and whose robust counterparts we shall study from a reverse perspective encompass the 
\begin{itemize}
    \item \textsf{Kreps-Yan property} (essential for the FTAP and studied in Section~\ref{sec:dominated})
    \item \textsf{Brannath-Schachermayer Bipolar Theorem} (see Section~\ref{sec:BS}, used in the literature for utility maximisation problems) \item \textsf{feasibility of aggregation procedures} (important for superhedging, dynamic arbitrage pricing, and utility maximization, see Sections~\ref{sec:aggregation:appl} and \ref{sec:logic:cor})
    \item \textsf{Fatou representations of convex monetary risk measures} (also relevant for the FTAP, see Section~\ref{sec:Grothendieck}).\end{itemize} 
This classification is an attempt at an axiomatic reply to the quoted question of \cite{BN} and can inform debates of specific models in robust finance.

Mathematically, our approach is closely related to \cite{surplus,Fatou} in its use of the so-called {\em $\CP$-quasi-sure ($\CP$-q.s.)\ order} as key. This order on function spaces is defined using the  upper envelope $\mbf c:=\sup_{\P\in\CP}\P(\cdot)$ of $\CP$.
While both $L^\infty_\PW$ equipped with the usual $\PW$-almost-sure order and its robust counterpart $\Linfty$ equipped with the $\CP$-q.s.\ order are Banach lattices, they differ substantially on the order level. 
A not overly deep, but crucial observation is the following. Consider \textit{any} measure $\mu\colon\CF\to[0,\infty]$ whose null sets are precisely the $\CP$-polar events. Such a measure always exists, for instance $\mu:=\sum_{\P\in\CP}\P$.\footnote{~$\mu:=\sum_{\P\in\CP}\P$ is defined by $\mu(A)=\sup\{\sum_{\P\in\CP'}\P(A)\mid \CP'\subset\CP\tn{ finite}\}$, $A\in\CF.$}
Then the $\CP$-q.s.\ and the \emph{$\mu$-almost-everywhere order} agree, and the robust space $\Linfty$ can always be seen as a classical $L^\infty_\mu$-space, where the measure $\mu$ is not $\sigma$-finite  if $\CP$ does not belong to category (b). While this point of view tends to be missing in the extant mathematical finance literature (whose natural objects of interest are probability measures), it has proved fruitful in robust statistics if the measure $\mu$ can be chosen to be reasonably well-behaved; cf.\ \cite{Luschgy}. 
Indeed, category (c') realises if $\mu$ and the alternative $\CQ$ to $\CP$ can be chosen such that each $\Q\in \CQ$ has a density with respect to $\mu$; cf.\ Lemma~\ref{lem:Luschgy}. 
The support of $\Q\in \CQ$ is the domain of positivity of that density. 
Being an alternative to $\CP$, the $\CQ$-q.s.\ and $\CP$-q.s.\ orders coincide. 

Intuitively, it is clear that the supportedness of each $\Q\in \CQ$ makes the $\CQ$-q.s.\ order, and thus the $\CP$-q.s.\ order, tractable in many ways. For instance, 
Section~\ref{sec:product} explains that function spaces defined over a set $\CP$ of class (S) necessarily are an infinite Cartesian product.
 Section~\ref{vol:uncertainty} relates it to several economically motivated case studies drawn from the literature.
Moreover, it will turn out to be {\em necessary} for the validity of the aforementioned theorems in a robust setting. Not least, the class (S) property serves as a bracket for parts of the extant literature on robust models: For instance, the models in \cite{Cohen,STZ} are of class (S). 
Section~\ref{sec:aggregation:appl} will demonstrate that indeed any model $\CP$ which admits {\em aggregation}, as is the case in \cite{Cohen,STZ}, must be of class (S). Aggregation means that any consistent family of random variables $(X^\PW)_{\PW\in \CP}$, can be aggregated to a single random variable $X$ in the sense that $\PW(X=X^\PW)=1$, $\PW\in \CP$.
Aggregation is also closely related to the existence of essential suprema, i.e.\ \textit{Dedekind completeness} of $\Linfty$. 
We will also show that the standing assumption of 
\cite{surplus,Fatou} that $\ca^*=\Linfty$ (where $\ca$ denotes the space of finite signed measures whose total variation measures do not charge $\CP$-polar events) is unnatural in {\bf ZFC} and better replaced by $\CP$ being of class (S) and $\Linfty$ being Dedekind complete; cf.\ Section~\ref{sec:logic:cor}.

Concluding this discussion and anticipating our results in a nutshell, we shall establish the following equivalences:

\medskip

\begin{tabular}{ccc}
 Kreps-Yan property &$\iff$&$\CP$ dominated\\
  Robust Brannath-Schachermayer Bipolar Theorem &$\iff$&$\CP$ of class (S)\\
Essential suprema and aggregation&$\iff$&$\CP$ of class (S) \& $\Linfty$ Dedekind complete\\
Fatou representation of risk measures&$\iff$&$\CP$ of class (S) \& $\Linfty$ Dedekind complete
\end{tabular}

\medskip

In view of the importance of Dedekind completeness/aggregation, Appendix~\ref{sec:enlargements} focuses on the question whether aggregation can be made possible by \textit{relaxing} the notion of measurability and \textit{enlarging} the underlying $\sigma$-algebra $\CF$. It turns out that this is only the case under specific circumstances.

Finally, we like to draw attention to a strong link between the fields of robust finance and \textit{robust statistics} which to our knowledge has not been sufficiently addressed in the extant literature. The consideration of nondominated sets of probabilistic models relevant for a statistical experiment goes back at least to \cite{Burkholder}, 
extending the theory of sufficiency developed by Halmos \& Savage. 
Beside the groundbreaking contribution of \cite{Sufficiency}, we refer to a rich strand of literature given by, for instance, \cite{Ghoshetal,Perlman,LeCam,Pitcher,Roy,Torgersen}. In particular, \cite{Luschgy} and the monographs~\cite{Fremlin2,Fremlin3} serve as important references for the present paper.

The paper is structured as follows: Section~\ref{sec:motivation} and Appendices~\ref{sec:prelim}--\ref{sec:appendix} contain explanatory and technical material. Section~\ref{sec:classS} discusses the class (S) property thoroughly, and Section~\ref{vol:uncertainty} illustrates its occurence in several economic case studies. 
Sections \ref{sec:BS} and \ref{sec:weak*} contain the characterisation of the robust counterparts of the well-known tools from mathematical finance mentioned above.  

\medskip

\section{Some notation and quasi-sure orders}\label{sec:motivation}

We shall rely heavily on lattice theory. A very brief overview is given in Appendix~\ref{sec:prelim}. For more information we refer to the monographs \cite{Ali,AliBurk, AliBurk2,MeyNie}.

\smallskip

\textsf{Set functions:} Throughout the paper $(\Omega,\CF)$ denotes an arbitrary measurable space. By $ba$ we denote the real vector space of all additive set functions $\mu\colon\CF\to\R$ with bounded total variation norm denoted by $TV$. $ba$ is a vector lattice when endowed with the setwise order: for $\mu,\nu\in ba$, $\mu\peq_\CF\nu$ holds if, for all $A\in\CF$, $\mu(A)\leq \nu(A)$. The triplet $(ba, \peq_\CF,TV)$ is a Banach lattice.

Given nonempty subsets $\mf S$ and $\mf T$ of $ba$, we say $\mf T$ \emph{dominates} $\mf S$ ($\mf S\ll\mf T$) if for all $N\in\CF$ satisfying $\sup_{\nu\in\mf T}|\nu|(N)=0$ we have $\sup_{\mu\in\mf S}|\mu|(N)=0$.
Here and in the following, $|\mu|\in ba_+$ denotes the total variation of $\mu$ with respect to $\peq_\CF$,
\[|\mu|(A)=\sup\{\mu(B)-\mu(A\backslash B)\mid B\in\CF,~B\subset A\}.\]
$\mf S$ and $\mf T$ are \emph{equivalent} ($\mf T\approx \mf S$) if $\mf S\ll\mf T$ and $\mf T\ll\mf S$.
For the sake of brevity, for $\mu\in ba$ we shall write $\mf S\ll\mu$, $\mu\ll\mf T$, and $\mu\approx\mf S$ instead of $\mf S\ll\{\mu\}$, $\{\mu\}\ll\mf T$, and $\{\mu\}\approx\mf S$, respectively.
 
The vector space of all countably additive signed measures with finite total variation, $ca$, is the space of all $\mu\in ba$ such that, for any sequence $(A_i)_{i\in\N}\subset\CF$ of pairwise disjoint events, 
\[\mu\Big(\bigcup_{i=1}^\infty A_i\Big)=\sum_{i=1}^\infty\mu(A_i).\]
$\Delta(\CF)$ denotes the set of probability measures on $(\Omega,\CF)$ and the letters $\mf P$, $\mf Q$, and $\mf R$ are henceforth used to denote nonempty subsets of $\Delta(\CF)$. Fixing such a set $\mf P$ we write $\mbf c$ for the induced upper probability $\mbf c\colon\CF\to[0,1]$ defined by $\mbf c(A)=\sup_{\P\in\CP}\P(A)$. An event $A\in \CF$ is {\em $\CP$-polar} if $\mbf c(A)=0$.  A property holds {\em $\CP$-quasi surely} (q.s.) if it holds outside a $\CP$-polar event. We set ${ba_{\mbf c}}:=\{\mu\in ba\mid \mu\ll\CP\}$ and $\ca:=ca\cap{ba_{\mbf c}}$.
$\ca_+:=\{\mu\in\ca\mid 0\peq_\CF\mu\}$ denotes the set of all measures $\mu\ll\CP$. Both $ca$ and $\ca$ are $TV$-closed bands within $ba$. Hence, $(\ca, \peq_\CF,TV)$ is a Banach lattice in its own right.

\smallskip

\textsf{Function spaces:} Consider the $\R$-vector space $\CL^0:=\mathcal L^0(\Omega,\CF)$ of all real-valued random variables $f\colon\Omega\to\R$ as well as its subspace $\mathcal N:=\{f\in\CL^0\mid\mbf c(|f|>0)=0\}$. The quotient space $\Lzero:=\CL^0/\mathcal N$ contains equivalence classes $X$ of random variables up to $\CP$-q.s.\ equality  comprising \emph{representatives} $f\in X$. The space $\Lzero$ carries the \emph{$\CP$-quasi-sure order} as natural vector space order: $X,Y\in\Lzero$ satisfy $X\peq Y$ if, for all $f\in X$ and all $g\in Y$, $f\le g$ $\CP$-q.s.\ 
In order to facilitate notation we suppress the dependence of $\peq$ on $\mbf c$ or $\CP$.
Each measurable function $f$ induces an equivalence class $[f]\in {L^0_{\mbf c}}$.  $({L^0_{\mbf c}},\peq)$ is a vector lattice, and for $X,Y\in {L^0_{\mbf c}}$, $f\in X$, and $g\in Y$, the minimum $X\wedge Y$ is the equivalence class $[f\wedge g]$ generated by the pointwise minimum $f\wedge g$, whereas the maximum $X\vee Y$ is the equivalence class $[f\vee g]$ generated by the pointwise maximum $f\vee g$. 
For an event $A\in\CF$, $\chi_A$ denotes the indicator of the event while $\ind_A:=[\chi_A]$ denotes the generated equivalence class in ${L^0_{\mbf c}}$.

An important subspace of ${L^0_{\mbf c}}$ which we will study thoroughly is the space $\Linfty$ of equivalence classes of $\CP$-q.s.\ bounded random variables, i.e.
\[\Linfty:=\{X\in {L^0_{\mbf c}}\mid \exists\, m>0:~|X|\peq m\ind_\Omega\}.\]
It is an ideal in ${L^0_{\mbf c}}$ and even a Banach lattice when endowed with the norm
\[\|X\|_\Linfty:=\inf\{m>0\mid |X|\peq m\ind_\Omega\},\quad X\in\Linfty.\]
If $\CP=\{\P\}$ is given by a singleton, we write $L^0_\P$ and $L^\infty_\P$ instead of $L^0_{\mbf c}$ and $\Linfty$. Also, the quasi-sure order in this case is as usual called \textit{almost-sure} order, and properties hold $\P$-almost surely ($\P$-a.s.). The spaces $L^0_\mu$ and $L^\infty_\mu$ for general measures $\mu$ are defined analogously. $\Lzeroplus$ and $\Linftyplus$ denote the positive cones of $\Lzero$ and $\Linfty$, respectively. 
At last, for $\emptyset\neq\CC\subset {L^0_{\mbf c}}$ and $A\in\CF$, we write
\[\ind_A\CC:=\{\ind_AX=[\chi_Af]\mid X\in\CC,\, f\in X\}.\]

\smallskip

\textsf{Supported measures:} The existence of a \textit{support} of a measure $\mu\in \ca$ will play a key role in the development of the present paper.  This concept is also known in statistics, see \cite[Definition 1.1]{Ghoshetal}.

\begin{definition}\label{def:support}Let $\CP\subset\Delta(\CF)$ be nonempty. 
\begin{enumerate}[(1)]
\item
A measure $\mu$ on $(\Omega,\CF)$ is \emph{supported} if there is an event $S(\mu)\in\CF$ such that 
\begin{enumerate}[(a)]
\item $\mu(S(\mu)^c)=0$;
\item whenever $N\in\CF$ satisfies $\mu(N\cap S(\mu))=0$, then $N\cap S(\mu)$ is $\CP$-polar.   
\end{enumerate}
The set $S(\mu)$ is called the \textit{(order) support} of $\mu$.
\item A signed measure $\mu\in\ca$ is supported if its modulus $|\mu|$ with respect to the setwise order $\peq_\CF$ is supported. 
\item We set 
\[
\sca:=\{\mu\in{ca_{\mbf c}}\mid\mu\tn{ supported}\},
\]
the vector lattice of all supported signed measures in $\ca$, and $\sca_+:=\sca\cap\ca_+$. 
\end{enumerate}
\end{definition}

It is easy to verify that if two events $S, S'\in\CF$ satisfy conditions (a) and (b) in Definition~\ref{def:support}(1), then $\chi_S=\chi_{S'}$ $\CP$-q.s., i.e.\ the symmetric difference $S\triangle S'$ is $\CP$-polar. 
The order support $S(\mu)$ is usually not unique as an event, but only unique up to $\CP$-polar events. In the following, $S(\mu)$ therefore denotes an arbitrary version of the order support. We emphasise that order supports are an order-theoretic concept and may not agree with the topological support of a measure as defined in~\cite[p.\ 441]{Ali}.

As illustration, consider a dominated set $\CP\subset\Delta(\CF)$. Let $\P^*\approx\CP$ be a probability measure. In particular, the $\CP$-q.s.\ and the $\P^*$-a.s.\ order agree. Then each $\P\in\CP$ is supported with $S(\P)=\{f>0\}$, where $f\in\tfrac{d\P}{d\P^*}$ is an arbitrary version of the Radon-Nikodym derivative $\frac{d\P}{d\P^*}$.  We shall later see that the existence of supports is in fact related to a generalisation of domination called \textit{majorisation}, cf.\ Definition~\ref{def:majorised}.

\smallskip

As mentioned in the introduction, we freely use results from mathematical statistics. There, the triplet $\CE:=(\Omega,\CF,\CP)$ is referred to as \textit{experiment}. Nondominated statistical experiments are comprehensively studied in \cite{Burkholder,Ghoshetal,Luschgy}, the references therein, or in the monographs~\cite{LeCam,Torgersen}, and we shall draw from these findings throughout the paper. 

\medskip

\section{Supported uncertainty}\label{sec:classS}

In this section we will introduce supported uncertainty, the so-called class \tn{(S)} property of $\CP$; see Definition~\ref{def:classS} below. 
Note, however, that usually not all measures are supported in the nondominated case. \begin{example}\label{ex:Lebesgue}
Consider the unit interval $\Omega=[0,1]$ equipped with the Borel-$\sigma$-algebra $\CF={\mathcal B}([0,1]))$. Let $\CP=\{\delta_\om\mid \om\in [0,1]\}$ be the set of all Dirac measures. Then the $\CP$-q.s.\ order is simply the pointwise order on $\Linfty$. However, the Lebesgue measure $\lambda$ is not supported: for any $S\in {\mathcal B}([0,1])$ with $\lambda(S^c)=0$ and for any countable subset $\emptyset\neq N\subset S$ we have $\lambda(N)=0$, but $\ind_N\neq 0$ in ${L^0_{\mbf c}}$.
\end{example}

\subsection{The class \tn{(S)} property}
We begin with the following simple fact about supportedness.

\begin{lemma}\label{lem:simple}Let $\CP\subset\Delta(\CF)$ be nonempty and let $\mu\in ca_+$. 
\begin{itemize}
\item[\tn{(1)}]$\mu$ is supported if and only if there is an event $S\in\CF$ such that $\mu(S^c)=0$ and $\ind_S\peq \ind_A$ in ${L^0_{\mbf c}}$ holds for all $A\in\CF$ with the property $\mu(A^c)=0$. In that case, $S$ is a version of $S(\mu)$.
\item[\tn{(2)}]Set $\CC:=\{\ind_A\mid A\in\CF,\, \mu(A^c)=0\}\subset \Linfty$. If $\inf\CC$ exists in $\Linfty$, then there is an event $S\in\CF$ such that 
\[\ind_{S}=\inf\CC.\]
In that case, $\mu$ is supported if and only if $S$ satisfies $\mu(S^c)=0$. 
\end{itemize}
\end{lemma} 
\begin{proof}\
\begin{enumerate}[(1)]
\item Suppose that $\mu$ is supported as defined in Definition~\ref{def:support}. Then any version $S(\mu)$ of that support satisfies $\mu(S(\mu)^c)=0$. Let $A\in\CF$ such that $\mu(A^c)=0$. Then $\mu(A^c\cap S(\mu))=0$ and thus $A^c\cap S(\mu)$ is $\CP$-polar. Therefore $$\ind_{S(\mu)} = \ind_{A \cap S(\mu)} + \ind_{A^c\cap S(\mu)} = \ind_{A \cap S(\mu)} \peq  \ind_{A}.$$ Conversely, let $S\in\CF$ be as described. Suppose $N\in\CF$ satisfies $\mu(N\cap S)=0$. 
As $\mu(N\cup S^c)=0$ and thus $\ind _S\peq \ind_{S\setminus N}$ by assumption on $S$ we infer that $\ind_{S}=\ind_{S\backslash N}$. Hence,  $\ind_{N\cap S}=\ind_{S}-\ind_{S\backslash N}=0$. By Definition~\ref{def:support}(1), $S$ is a version of the support.
\item The existence of the set $S\in\CF$ is a direct consequence of Lemma~\ref{lem:sets}. 
If $\mu$ is supported and $S(\mu)$ is its support, then $\ind_{S(\mu)}\peq \inf \CC=\ind_S$ holds by (1). As $\ind_{S(\mu)}\in\CC$, we also have $\ind_S\peq\ind_{S(\mu)}$. Hence, $\ind_S=\ind_{S(\mu)}$ has to hold. Conversely, if $\mu(S^c)=0$, $S$ is the support of $\mu$ by (1).
\end{enumerate}
\end{proof}

Example \ref{ex:Lebesgue} suggests that asking for supportedness of all measures in $\ca$ is a requirement which is too strong in the context of nondominated models. We thus introduce the weaker class (S) property of $\CP$.

\begin{definition}\label{def:classS}
Let $\CP\subset\Delta(\CF)$ be nonempty. 
$\CP$ is \emph{of class \tn{(S)}} if there is an equivalent set of probability measures $\CQ\approx\CP$ such that each $\QW\in \CQ$ is supported, i.e.\ $\CQ\subset \sca$. In that case we call $\CQ$ a \textit{supported alternative} to $\CP$.
\end{definition}

As mentioned above, the class (S) property is closely related to a generalised form of domination called majorisation: 

\begin{definition}\label{def:majorised} Let $\CP\subset\Delta(\CF)$ be nonempty.
$\CP$ is \emph{majorised} by a (not necessarily $\sigma$-finite) measure $\mu$ on $(\Omega,\CF)$ if each $\P\in\CP$ has a $\mu$-density, i.e.\ for all $\P\in\CP$ there is a non-negative measurable function $f\colon\Omega\to[0,\infty)$ such that 
\[\forall\,A\in\CF:~\P(A)=\int_A f\,d\mu.\]
\end{definition}

\begin{remark}
 The class (S) property is stable under equivalence. Indeed, if $\CP$ is of class (S) and another set $\CP'\subset\Delta(\CF)$ satisfies $\CP\approx\CP'$, then $\CP'$ is of class (S) as well.
 This implication does not hold for the property of majorisation: $\CP$ being majorised does {\em not} imply that $\CP'$ is majorised. 
 This observation is also motivated by the tendency of mathematical finance applications to work on equivalence classes of probability measures.
\end{remark}

The link between the class (S) property and majorisation is the following: 

\begin{lemma}\label{lem:Luschgy}
For nonempty $\CP\subset\Delta(\CF)$, the following are equivalent: 
\begin{itemize}
\item[\tn{(1)}] $\CP$ is of class \tn{(S)}.
\item[\tn{(2)}] There is an equivalent set of probability measures $\CQ\approx\CP$ which is majorised. 
\end{itemize}
If $\CQ\subset\Delta(\CF)$ is as in \tn{(2)}, $\sca$ is the band generated by $\CQ$ both in ${ca_{\mbf c}}$ and in $ca$. 
\end{lemma}
\begin{proof}
The equivalence of (1) and (2) follows from the fact that $\CQ$ is majorised if and only if each $\QW\in\CQ$ is supported, see \cite[Remark 1.1(a)]{Ghoshetal} or \cite[Theorem 1]{Luschgy}. 
We now prove the last assertion. To this end, consider the ambient space $ca$. Let $\mbf{\widetilde c}$ denote the upper probability generated by $\CQ$. By \cite[Theorem 1]{Luschgy},  
\[\tn{band}(\CQ)=\{\mu\in{ca_{\mbf{\widetilde c}}}\mid|\mu|\tn{ has an order support}\}=sca_{\mbf{\widetilde c}},\]
where band$(\CQ)$ denotes the band generated by $\CQ$ in $ca$. 
As $\CQ\approx\CP$, we have both $ca_{\mbf{\widetilde c}}=ca_{\mbf{c}}$ and  $sca_{\mbf{\widetilde c}}=\sca$, and the claimed identity is proved. Now ${ca_{\mbf c}}$ is a band in $ca$ itself (\cite[Lemma 1]{Luschgy}) and $\CQ\subset\sca\subset{ca_{\mbf c}}$. Hence, the same equality holds for the band in ${ca_{\mbf c}}$ generated by $\CQ$.
\end{proof}

Given nonempty $\CP\subset\Delta(\CF)$,a subset $\mf T\subset\sca_+$ is a \emph{maximal disjoint system} if 
\begin{enumerate}[(i)]
\item $0\notin\mf T$, 
\item for all $\mu,\nu\in\mf T$ with $\mu\neq \nu$ we have that $\mu\wedge\nu=0$,
\item $\nu=0$ whenever $\nu\in\sca_+$ satisfies $\mu\wedge\nu=0$ for all $\mu\in\mf T$.
\end{enumerate}
If $\sca$ is nontrivial---for instance, if $\CP$ is of class (S) (Lemma~\ref{lem:Luschgy})---a maximal disjoint system in $\sca_+$ exists by Zorn's Lemma. 
Such maximal disjoint systems play a fundamental role for the results of \cite{Luschgy}, and we will use them in the guise of the next lemma.

\begin{lemma}\label{lem:disjoint:sup:alt}
Suppose $\CP$ is of class \tn{(S)}. Then there exists a supported alternative $\CQ\approx\CP$ such that, for all $\QW\neq\QW'$, $\QW\wedge\QW'=0$ holds in $\ca$.
Equivalently, $\ind_{S(\QW)}\wedge\ind_{S(\QW')}=0$ in $\Linfty$. 
\end{lemma}
\begin{proof}
Let $\mf T\subset\sca_+$ be a maximal disjoint system. Set $\CQ:=\{\mu(\Omega)^{-1}\mu\mid\mu\in\mf T\}$. For $\Q\in\CQ$, we consider the $\Q$-expectation $\E_\Q[\cdot]$ on $\Linfty$, which is order continuous by Proposition~\ref{prop:(A1)-(A3)}(3) and whose so-called \textit{carrier} is given by $C(\E_\Q[\cdot])=\ind_{S(\Q)}\Linfty$; cf.\ Appendix~\ref{sec:prelim} and Proposition~\ref{prop:(A1)-(A3)}(2). 
For $\Q,\Q'\in\CQ$ with $\Q\neq\Q'$, the fact that $\Q\wedge\Q'=0$ now is equivalent to $\ind_{S(\Q)}\wedge\ind_{S(\Q')}=0$ by 
 \cite[Theorem 1.81]{AliBurk}. 
\end{proof}
We will refer to $\CQ$ from Lemma~\ref{lem:disjoint:sup:alt} as a  \textit{disjoint supported alternative}. 
The existence of a disjoint supported alternative has appeared in \cite{Cohen} and studies based thereon as the so-called \textit{Hahn property} of $\CP$. 
A set $\CP$ of probability measures has the \textit{Hahn property} if $\CP$ is of class \tn{(S)} and there is a disjoint supported alternative $\CQ$ to $\CP$ with the property $\CF(\CQ)=\CF(\CP)$ which admits a family $(S(\Q))_{\Q\in\CQ}$ of {\em pairwise disjoint} versions of the associated supports. Here, we use the following notation: Given a set $\mf S$ of measures on $(\Omega,\CF)$, the \emph{$\mf S$-completion} of $\CF$ is the $\sigma$-algebra
\begin{equation}\label{eq:defenlarge}\CF(\mf S):=\bigcap_{\mu\in\mf S}\sigma(\CF\cup\mf n(\mu)),\end{equation}
where $\mf n(\mu):=\{A\subset\Omega\mid \exists\,N\in\CF:~\mu(N)=0,\, A\subset N\}$ denotes the $\mu$-negligible sets, $\mu\in\mf S$. 
A set $A\subset\Omega$ belongs to $\CF(\mf S)$ if and only if for all $\mu\in\mf S$ there is $B_\mu\in\CF$ such that $A\triangle B_\mu=(A\backslash B_\mu)\cup (B_\mu\backslash A)\in\mf n(\mu)$. 

The Hahn property of $\CP$ implies that $\CP$ is of class \tn{(S)} by definition. Conversely, Lemma~\ref{lem:disjoint:sup:alt} shows that if $\CP$ is of class \tn{(S)}, then $\CP$ satisfies a weak kind of Hahn property. However, $\CP$ being of class (S) does not imply the Hahn property (cf.\ \cite[216E]{Fremlin2}).
In the statistics literature, the Hahn property is closely related to \textit{decomposable, $\Sigma$-finite}, or \textit{$\Sigma$-dominated} experiments; cf.\ \cite[p.\ 185]{Luschgy} and \cite[p.\ 7]{Torgersen}.

\smallskip
 
\subsection{Examples for the class (S) property}\label{vol:uncertainty}

We shall now discuss and verify the class (S) property in a number of prominent case studies from mathematical finance. While the purpose is illustration, let us once again emphasise that we do not endorse the class (S) property as an axiom robust models {\em should} satisfy. We mostly identify it as a {\em consequence} of the validity of robust counterparts of well-known tools in mathematical finance. 

\smallskip

\subsubsection{Class (S), product spaces, and  financial models}\label{sec:product}
In two recent papers the FTAP and the pricing-hedging duality under uncertainty are approached both in discrete-time markets \cite{Ch20} and continuous-time markets with frictions \cite{CFR21}. 
Given a filtered probability space $(\Omega,\CF, (\CF_t)_{t\in\mathbb T},\PW)$, the price process at time $t\in\mathbb T$ is modelled as a potentially infinite-dimensional vector 
$$\mathbf{S}_t=(S_t^{\theta})_{\theta\in\Theta}\in \prod_{\theta\in \Theta} L^0_\P,$$
where $S_t^{\theta}$ is $\CF_t$-measurable for any $\theta$. 
$\Theta$ represents the set of parameters describing the uncertainty in the market model.  

We deem important to illustrate that this construction from \cite{Ch20,CFR21} indeed falls in the class (S) framework. To this end, one enlarges the underlying measurable structure and considers
$$\begin{array}{l}\widetilde\Omega:=\Omega\times\Theta\\
\Pi_{t}=\big\{A^{\theta}\times\{\theta\}\subset\widetilde\Omega\mid\theta\in\Theta,\, A^{\theta}\in \CF_{t}\big\}\cup\{\emptyset\}\quad(t\in\mathbb T),
\\ \Pi=\big\{A^{\theta}\times\{\theta\}\subset\widetilde\Omega\mid\theta\in\Theta,\, A^{\theta}\in \CF\big\}\cup\{\emptyset\}.\end{array}$$
$\Pi_{t},\Pi$ are $\pi$-systems generating $\sigma$-algebras $\widetilde{\CF}_{t}$ and $\widetilde{\CF}$, respectively.
Using Dynkin's $\pi$-$\lambda$ Theorem, we can define 
\[\P^\theta(B):=\P(\{\om\in\Omega\mid (\om,\theta)\in B\}),\quad B\in\widetilde\CF.\]
\begin{lemma} 
The family $\CP=\{\PW^{\theta}\mid \theta\in\Theta\}$ on $(\widetilde{\Omega},\widetilde{\CF})$ is of class \tn{(S)}. 
\end{lemma}
The proof of the previous lemma is straightforward and follows the idea adopted in~\cite[Corollary 5.7.14]{Torgersen}: indeed the measurable sets $\Omega^{\theta}=\Omega\times \{\theta\}$ are the supports of each measure $\PW^{\theta}$, as by definition $\PW^{\theta}(\Omega^{\theta})=1$ and $\Omega^{\theta}\cap \Omega^{\theta'}=\emptyset$ for $\theta\neq \theta'\in\Theta$.   

Hence, the financial model given by the filtered probability space $(\Omega,\CF,(\CF_t)_{t\in\mathbb T},\P)$, the parameter set $\Theta$, and the price process $\mathbf{S}$ can equivalently be replaced by the new  measurable space $(\widetilde\Omega,\widetilde\CF)$, the filtration $(\widetilde{\CF}_t)_{t\in\mathbb T}$, and the set $\CP=\{\P^\theta\mid\theta\in\Theta\}$ of relevant probability measures (and of course the price process is easily redefined on this new structure). 
The preceding construction is a special case of a procedure that is always possible in the class (S) framework. Recall from Lemma~\ref{lem:disjoint:sup:alt} that there is a disjoint supported alternative $\CQ$ to $\CP$. Consider the vector space
\begin{equation}\label{product space}\mathcal Y:=\big\{\mathbf X=(X_\Q)_{\Q\in\CQ}\big| \sup_{\Q\in\CQ}\|X_\Q\|_{L^\infty_\Q}<\infty\big\}\subset \prod_{\Q\in\CQ}L^\infty_\Q
\end{equation}
and set $\mathbf X\tle\mathbf Y$ if $X_\Q\le Y_\Q$ $\Q$-a.s.\ holds for all $\Q\in\CQ$. Then $(\CY,\tle)$ is in fact a vector lattice. 
If we define $j_\Q\colon\Linfty\to L^\infty_\Q$ by setting $j_\Q(X)$ to be the equivalence class generated in $L^\infty_\Q$ by any representative $f\in X$, we obtain a strictly positive lattice homomorphism
\[J\colon \Linfty\to\CY,\quad X\mapsto(j_\Q(X))_{\Q\in\CQ}.\]
Given that $\CQ$ is a disjoint supported alternative, one can verify that $J(\Linfty)$ is order dense and majorising in $\CY$, notions introduced in Appendix~\ref{sec:prelim}. 
In sum, if $\CP$ is of class (S), $\Linfty$ is lattice isomorphic to a subspace of a product space. If $\CP$ is not dominated, the latter has uncountably many coordinates. 

\smallskip

\subsubsection{Volatility uncertainty}\label{sec:voluncertainty}
In continuous-time financial models one of the most relevant sources of uncertainty is related to the estimation of the volatility of price processes. 
We illustrate here that the model of uncertain volatility discussed in \cite{Cohen,STZ} falls in the class (S) framework. To this end, let $\P_0$ be the Wiener measure on the Wiener space $\Omega$ of continuous functions $\omega\colon\R_+\to\R$ with $\om(0)=0$, so that the canonical process  $B:=(B_t)_{t\in\R_+}$ defined by $B_t(\om)=\om(t)$, $t\in\R_+$, $\om\in\Omega$, is a standard Brownian motion under $\P_0$ with respect to the natural filtration $\mathbb F=(\CF_t)_{t\ge 0}:=(\sigma(B_s\mid 0\leq s\leq t))_{t\ge 0}$. 
\cite{Ka95} proves that there is an $\mathbb F$-adapted process $\langle B\rangle$ such that under each probability measure $\P$ on $(\Omega,\CF)$ with respect to which $B$ is a local martingale, $\langle B\rangle$ agrees with the usual $\P$-quadratic variation of $B$ $\P$-a.s.\ 
Further, let $\CP^{obs}$ denote the set of all probability measures $\P$ under which the canonical process is a local martingale and for which $\P$-a.s.\ $\langle B\rangle_\cdot$ is absolutely continuous in $t$ and takes positive values. 
The set $\CP^{obs}$, however, is too large for the considerations made in \cite{STZ} for various reasons. 
In particular, when working under $\CP^{obs}$, it is impossible to  establish a one-to-one correspondence between volatility processes and probability measures. Since the uncertainty in this case stems from uncertainty about the right volatility process such an identification is needed.
This problem is overcome by considering  a nonempty set $\mathcal V$ of volatility processes $\sigma$ such that the stochastic differential equation under the Wiener measure $\P_0$
\begin{equation}\label{eq:SDE}dX_t=\sigma_t(X)dB_t\end{equation}
has \textit{weak uniqueness} in the sense of \cite[Definition 4.1]{STZ}. Given $\sigma$, this admits the selection of a unique $\P^\sigma$ such that $dB_t= \sigma_t(B)dW_t^{\sigma}$ $\P^\sigma$-a.s., where $W^{\sigma}$ is a $\P^\sigma$-standard Brownian motion. 
The set of probability measures $\CP\subset\CP^{obs}$ we obtain from such a set $\mathcal V$ is usually nondominated. For each such $\sigma\in\mathcal V$, the event 
\begin{equation}\label{eq:S(Psigma)}S(\P^\sigma):=\{\om\in\Omega\mid \forall\,t\in\QW_+:~\langle B\rangle_t(\om)=\int_0^t\sigma_s^2(\om)ds\}\end{equation}
satisfies $\P^\sigma(S(\P^\sigma))=1$. 
However, note that these events generally fail condition (b) in Definition~\ref{def:support}(1) if $\mathcal V$ contains more complex  $\omega$-dependent volatility processes $\sigma$. The main difficulty is that control of their intersections is tedious. 
This is precisely the reason why in \cite{STZ} the authors further thin out the set of admissible volatility processes to sets $\mathcal V$ of \textit{separable diffusion coefficients}; see \cite{STZ} for the details. 
As already noticed in \cite{Cohen} the resulting set $\CP:=\{\PW^\sigma\mid \sigma\in\mathcal V\}$ satisfies the Hahn property discussed above. Hence, $\CP$ is of class (S). More precisely, the events $S(\P^\sigma)$ in \eqref{eq:S(Psigma)} can be shown to be order supports of the $\PW^\sigma$, $\sigma\in\mathcal V$, so that $\CP$ is in fact its own supported alternative. The main setting of interest in \cite{STZ} thus embeds in our framework of probabilities of class (S). 

\smallskip

\subsubsection{Innovation}  Our next case study deals with innovation economics and its relation to Knightian uncertainty. We consider a model suggested in \cite[p.\ 2247]{Contracts}, which posits that innovation adds newly explored states given by a measurable space $(S_n,\CF_n)$ to a set of known states given by the measurable space $(S_o,\CF_o)$, $S_o,S_n$ being two nonempty sets. The combined state space $\Omega$ has the shape
\[\Omega=S_o\times S_n\]
 and is endowed with a product-$\sigma$-algebra $\CF=\CF_o\otimes\CF_n$.
 While the agents in the model are subject to mere risk on the known states, the newly discovered states add uncertainty. 
In order to capture this phenomenon with a set $\CP$ of relevant probability measures, one uses the projection $X_i\colon \omega=(s_1,s_2)\mapsto s_i$, $i=1,2$, and a probability measure $\pi$ on $(S_o, \CF_o)$ that we interpret as the first marginal. One then sets 
 \[\CP:=\{\P\in\Delta(\CF)\mid \P\circ X_1^{-1}=\pi\}.\]
 Suppose first that $S_o$ is discrete, that $\pi$ has full support on $S_o$, and that $\CF_n$ contains all singletons.
 Consider
 \begin{equation}\label{eq:defQ}\CQ:=\{\pi\otimes\delta_{s_2}\mid s_2\in S_n\}\subset\CP.\end{equation}
 Note that, for fixed $s_*\in S_n$ and $\Q:=\pi\otimes \delta_{s_*}$, $\Q(N)=0$ holds if and only if there is no $s_1\in S_o$ such that $(s_1,s_*)\in A$. 
 Hence, $\sup_{\Q\in\CQ}\Q(N)=0$ holds for $N\in\CF$ if and only if $N=\emptyset$. 
 This implies $\CP\approx\CQ$. One verifies that $\pi\otimes \delta_{s_*}\in \CQ$ is supported by $S_o\times \{s_\ast\}$, and thus $\CP$ is of class (S) with supported alternative $\CQ$ in this case. 
 
Similarly, we can consider a case in which $S_o$ is not necessarily discrete and $\pi$ is arbitrary, but $\CF_n$ still contains all singletons. To a certain degree, we limit uncertainty by postulating independence of coordinate projections $X_1$ and $X_2$. 
This leads to a redefinition of the set of relevant probability measures as
\[\CP:=\{\P\in\Delta(\CF)\mid \mbox{$\P$ is a product measure and}\, \P\circ X_1^{-1}=\pi\}.\] 
$\CQ$ from \eqref{eq:defQ} again satisfies  
$\CP'\approx \CQ$. Indeed, $\sup_{\Q\in\CQ}\Q(N)=0$ if and only if all sections $N_{s_2}:=\{s_1\in S_o\mid (s_1,s_2)\in N )\}$ are $\pi$-nullsets. Hence, Fubini's Theorem implies that $\P(N)=0$ for any $\P\in \CP$. Thus also in this case $\CP$ is of class (S). 
 
 \smallskip
 
 \subsubsection{Typical paths} The third case study in which we can---rather trivially---verify the class (S) property follows a recent strand of literature \cite{Prediction,B+al16,HouObloj} on superhedging that is inspired by \cite{Mykland}. It hinges on the very definition of a superhedge. A notion which is independent of any concrete choice of probability model is usually referred as ``pointwise": the superhedging strategy has to dominate in every state, or under every realisation of a concrete path. 
 As discussed in \cite{Prediction,HouObloj}, such a price tends to be unreasonably high and ignores more precise information an investor might have.
  
 As an alternative, one considers an event $\Xi\in\CF$ within a state space $(\Omega,\CF)$ (usually a path space) of ``typical" or relevant states/paths. The superhedging is then only demanded in states $\om\in\Xi$ belonging to the prediction set. As elaborated in \cite{Prediction}, a canonical probabilistic description of this situation is given by 
 \[\CP:=\{\P\in\Delta(\CF)\mid \P(\Xi)=1\}.\]
Let $\CA$ be the collection of all atoms of $\CF$, i.e.\ all events $A\subset\Xi$ such that each measurable subset $B\subset A$ satisfies $B\in\{\emptyset,A\}$. Atoms are identical or disjoint. Let us also impose the following technical assumption: 
 \begin{assumption}
    For all $\emptyset\neq B\in\CF$ there is $\emptyset \neq A\in\CA$ such that $A\subset B$.
 \end{assumption}
The assumption is satisfied if, for instance, all singletons are measurable. A disjoint supported alternative to $\CP$ is now given by selecting $\om_A\in A$, $A\in\CA$, and setting $\CQ:=\{\delta_{\om_A}\mid A\in\CA,~A\subset\Xi\}$.

\smallskip

\subsubsection{Robust binomial model}\label{sec:BN}

One of the most prominent models for robust discrete-time financial markets is the structure proposed in \cite{BN} whose basics we briefly recall in the following. Consider a finite time horizon $T\in\N$, time points $t\in\mathbb T:=\{0,...,T\}$, and a space $\Omega=\widetilde\Omega^T$ of paths within a Polish space $\widetilde\Omega$. 
The uncertainty structure is tree-like; relevant probability measures are prescribed for each time point $t\in\mathbb T$ and each historical path up to time $t$. 
More precisely, set $\Omega_0$ to be a singleton and 
$\Omega_t:=\widetilde\Omega^t$, $0\neq t\in\mathbb T$. 
Let $\CF_t$ denote the universal completion of the Borel $\sigma$-algebra $\CB(\Omega_t)$.\footnote{~The universal completion of $\CB(\Omega_t)$ is the $\sigma$-algebra obtained by choosing $\CF=\CB(\Omega_t)$ and $\mf S=\Delta\left(\CB(\Omega_t)\right)$ in \eqref{eq:defenlarge}.} 
For $t=0,\dots,T-1$ and $\om\in\Omega_t$ fixed,  $\CP_t(\omega)\subset\Delta\big(\mathcal B(\widetilde \Omega)\big)$ is a prescribed convex set of probability measures on the node $(t,\omega)$. 
In this context the main difficulty is proving that \[\text{graph}(\CP_t)=\{(\omega,\P)\mid \omega\in \Omega_t,\; \P\in \CP_t(\omega)\}\]
is analytic in order to apply measurable selection techniques. To our knowledge, this requirement is verified only in few concrete examples in the literature. We concentrate our attention to one of them, namely the robust binomial model presented in \cite{Car}.\footnote{~We conjecture that it is not possible to show that the general structure in \cite{BN} is of class (S). In any case, this question is beyond the scope of the present paper.}

In this example, we set $\widetilde\Omega=(0,\infty)$, $\mathbb T= \{0,\ldots, T\}$, $\Omega=(0,\infty)^T$, $\Omega_0=\{1\}$, and  $\Omega_t=(0,\infty)^t$, $0\neq t\in\mathbb T$.
Any $\omega\in \Omega$ (resp.\ $\omega\in\Omega_t$) is represented by a tuple $\omega=(x_1,\ldots,x_T)$ (resp.\ $\omega=(x_1,\ldots,x_t)$). We introduce a price process $(S_t)_{t\in \mathbb T}$ defined by $S_0=1$ and $S_{t+1}=S_t\cdot Y_{t+1}$. Here $Y_{t+1}:(0,\infty)\to (0,\infty)$ is a bijective map for all $t\in \{0,\ldots,T-1\}$. $\Delta:=\Delta\left(\mathcal B((0,\infty))\right)$ abbreviates the set of Borel probability measures on $(0,\infty)$. For any $t\in \{0,\ldots,T-1\}$, let $u_t,U_t,d_t,D_t,\pi_t,\Pi_t:\Omega_t\to[0,\infty)$ be Borel measurable random variables such that for any $\omega\in \Omega_t$ the following inequalities hold: 
\begin{enumerate}[(i)]
\item $ 0<\pi_t(\omega)\leq \Pi_t(\omega)< 1$; 
\item $d_t(\omega)\leq D_t(\omega)$ and $u_t(\omega)\leq U_t(\omega)$; 
\item $0<d_t(\omega) < 1 < U_t(\omega)$.
\end{enumerate}
These requirements allow to define a random set $E_t(\cdot):=[u_t(\cdot),U_t(\cdot)]\times [d_t(\cdot),D_t(\cdot)]\times [\pi_t(\cdot),\Pi_t(\cdot)]$, $t\in\mathbb T$. In turn, the probability measures in question are constructed as follows. 
For any $\omega\in\Omega_t$, consider the set of binomial laws $\mathcal{L}_{t+1}(\omega)=\{\pi \delta_u+(1-\pi)\delta_d\mid (u,d,\pi)\in E_t(\omega)\}$ respecting the constraints given by $E_t(\omega)$.
Define 
\begin{center}$\CQ_{t+1}(\omega)=\{\QW\in \Delta\mid \QW\circ Y_{t+1}^{-1} \in \mathcal{L}_{t+1}(\omega))\}$,\end{center}
i.e.\ the probabilities under which $Y_{t+1}$ has a Bernoulli-like distribution from  $\mathcal{L}_{t+1}(\om)$. Next let $\CP_{t+1}(\omega) =\textnormal{conv}(\CQ_{t+1}(\omega))$. 
\cite[Lemma 4.3]{Car} shows that  both sets $\textnormal{graph}(\CQ_{t+1})$ and $\textnormal{graph}(\CP_{t+1})$ are analytic, $t=0,\dots,T-1$, matching the main requirement necessary to apply the results in \cite{BN}. In particular $\CQ_{t+1},\CP_{t+1}$ admit measurable selectors, i.e., there are universally measurable stochastic kernels $Q_{t+1}:\Omega_t\to \Delta$ and  $P_{t+1}:\Omega_t\to \Delta$ such that $Q_{t+1}(\omega)\in \CQ_{t+1}(\omega)$, $P_{t+1}(\omega)\in \CP_{t+1}(\omega)$ for any $\omega\in \Omega_{t}$. At last, the family of multiperiod probabilities $\CP$ or $\CQ$ on $\Omega:=\Omega_T$ are introduced as 
\begin{align*}
\CP&:=\{\PW=P_1\otimes P_2\otimes \ldots \otimes P_{T}\mid P_t(\cdot)\in \CP_{t}(\cdot),\, t\in\mathbb T\},\\
\CQ&:=\{\QW=Q_1\otimes Q_2\otimes \ldots \otimes Q_{T}\mid Q_t(\cdot)\in \CQ_{t}(\cdot),\, t\in\mathbb T\}, 
\end{align*}
where the measures $\PW=P_1\otimes P_2\otimes \ldots \otimes P_{T}$ are defined as
\[\PW(A)=\int_{0}^{\infty}\ldots \int_{0}^{\infty}\ind_{A}(x_1,\ldots,x_T)P_{T}(x_1,\ldots,x_{T-1}; dx_T)\dots\cdot P_{2}(x_1; dx_2)\cdot P_1(dx_1),\quad A\in \mathcal B(\Omega).\]
The following proposition illustrates how the robust binomial model falls in the setup of the present paper (see Appendix \ref{proof:bin} for a proof).
\begin{proposition}\label{rob:bin}
Let $\CF:=\CF_T$ be the universal completion of the Borel $\sigma$-algebra $\mathcal B(\Omega)$. The family $\CP$ on $(\Omega,\CF)$ is of class \tn{(S)} with $\CQ$ being its supported alternative.
\end{proposition}

\subsection{$\bf \sca\neq\ca$ is often the case}\label{sec:implications:axioms} 

In this section we demonstrate that $\sca\neq\ca$ is the case in a broad class of examples over Polish spaces. We 
also discuss the relation of this observation to the procedure of aggregating Bayesian experts in \cite{Amarante}. 
The assumption that the underlying $\Omega$ is Polish is common in financial applications, a fact also illustrated by Section~\ref{sec:BN}. 
Recall that if $\Omega$ is Polish and $\CF$ denotes its Borel-$\sigma$-algebra, then also  $\Delta(\CF)$ is Polish (\cite[Theorem 15.15]{Ali}).
As a preparation for Proposition~\ref{prop:Polish} below, recall that a subset $P$ of a Polish space is \textit{perfect} if it is closed and if, for every $\sigma\in P$, the closure of $P\setminus\{\sigma\}$ is $P$.

\begin{proposition}\label{prop:Polish}
Let $\Omega$ be Polish, $\CF$ be the Borel-$\sigma$-algebra on $\Omega$, and $\CP\subset\Delta(\CF)$ be nonempty. Suppose there is a set $\emptyset\neq\CR\subset\Delta(\CF)\cap\sca$ which is  perfect in $\Delta(\CF)$, and which satisfies
\[\forall\,\QW,\QW'\in\CR:~\QW\neq\QW'~\Longrightarrow~\ind_{S(\QW)}\wedge\ind_{S(\QW')}=0.\]
Then there is a probability measure $\mu\in{ca_{\mbf c}}\backslash \sca$.
The assertion also holds if perfectness of $\CR$ is replaced by the assumption that $\CR$ is an uncountable Borel or analytic set, respectively.
\end{proposition}
\begin{proof}
Let $\CR$ be the mentioned subset of $\Delta(\CF)\cap \sca$. Then there is a continuous injective map $\QW_\bullet:\{0,1\}^\N\to\Delta(\CF)$ with $\{\QW_\sigma\mid\sigma\in\{0,1\}^\N\}\subset\CR$; cf.\ \cite[Theorem 6.2]{Kechris}.
The Cantor space $\{0,1\}^\N$ is tacitly assumed to be endowed with the discrete product topology which is Polish.
For any $E\in\CF$, the function $\{0,1\}^\N\ni\sigma\mapsto\QW_\sigma(E)$, is Borel measurable as a composition of the continuous map $\QW_\bullet$ and the Borel measurable function $\Delta(\CF)\ni \mu\mapsto \mu(E)$ (\cite[Lemma 15.16]{Ali}). Let $\pi$ be any non-atomic Borel probability measure on $\{0,1\}^\N$, whose existence is guaranteed by \cite[Theorem 12.22]{Ali}. Consider 
\begin{equation}\label{eq:measure}\mu\colon\CF\to[0,1],\quad E\mapsto\int_{\{0,1\}^\N}\QW_\sigma(E)\pi(d\sigma).\end{equation}
$\mu$ is a probability measure dominated by $\CP$. Assume for contradiction that $\mu$ is supported. Let $\sigma\in\{0,1\}^\N$ be arbitrary. For all $\{0,1\}^\N\ni\sigma'\neq \sigma$, we have $\ind_{S(\QW_\sigma)}\wedge\ind_{S(\QW_{\sigma'})}=0$ in $\Linfty$, whence $\QW_{\sigma'}(S(\QW_\sigma))=0$ follows. By \eqref{eq:measure},
\[\mu(S(\mu)\cap S(\QW_{\sigma}))=\QW_{\sigma}\left(S(\mu)\cap S(\QW_\sigma)\right)\pi(\{\sigma\})=0.\]
The definition of a support shows that $\ind_{S(\mu)\cap S(\QW_{\sigma})}=0$ in $\Linfty$. As $\sigma$ was chosen arbitrarily, $\QW_{\sigma}(S(\mu))=\QW_{\sigma}(S(\mu)\cap S(\QW_\sigma))=0$ for all $\sigma\in\{0,1\}^\N$, contradicting $\mu(S(\mu))=1$.
At last, any uncountable Borel or analytic subset of $\Delta(\CF)$ contains a nonempty perfect set; see \cite[Theorems 13.6 \& 29.1]{Kechris}. 
\end{proof}

We now illustrate Proposition~\ref{prop:Polish} in the context of volatility uncertainty as in Section~\ref{sec:voluncertainty}.

\begin{example}\label{ex:big}
Assume that the set $\mathcal V$ of volatility processes in Section~\ref{sec:voluncertainty} contains  
the constant volatility processes $\kappa:\R_+\ni t\mapsto \kappa$, $0<\kappa_1\le \kappa\le \kappa_2$. 
Set $\CP:=\{\P^\sigma\mid \sigma\in\mathcal V\}$. 
We recall that each $\P^\sigma$ is supported with order support
\begin{center}$S(\P^\sigma):=\{\om\in \Omega\mid\forall\, t\in\QW_+:~\langle B\rangle_t(\om)=\int\limits_0^t\sigma_s(\om)^2ds\}\in\CF$,~$\sigma\in\mathcal V$,\end{center} so that $S(\P^\kappa):=\{\om\in \Omega\mid\forall\, t\in\QW_+:~ \langle B\rangle_t(\om)=\kappa^2t \}$, $\kappa\in [\kappa_1,\kappa_2]$.
One easily verifies that the set $\CR:=\{\P^\kappa\mid\kappa\in [\kappa_1,\kappa_2]\}\subset\Delta(\CF)\cap\sca$
is uncountable and closed in the topology of weak convergence on $\Delta(\CF)$. 
Moreover, for $\kappa_1\leq\kappa<\kappa'\leq\kappa_2$, $\ind_{S(\P^\kappa)}\wedge\ind_{S(\P^{\kappa'})}=0$ $\CP$-q.s.\ We thus are precisely in the situation of Proposition~\ref{prop:Polish}. Let $\pi$ be any non-atomic probability measure on $[\kappa_1,\kappa_2]$.
Arguing as in the proof that the measure in \eqref{eq:measure} is not supported and using Proposition~\ref{prop:(A1)-(A3)}, one shows \textit{inter alia} that for any choice of a strictly increasing {\em utility function} $u\colon\R\to\R$ the functional 
\begin{equation}\label{eq:amarante}\phi\colon\begin{array}{l}\Linfty\to\R,\\
X\mapsto\int_{\kappa_1}^{\kappa_2}\E_{\P^\kappa}[u(X)]\pi(d\kappa),\end{array}\end{equation}
is not order (semi)continuous.
\end{example}

Functionals of shape~\eqref{eq:amarante} have a natural interpretation in the framework of \cite{Amarante}, where a set of probability measures $\CP$ models the opinions of {\em Bayesian experts} concerning the probabilities of relevant events. These experts are consulted by a decision maker (DM) with utility function $u\colon \R\to\R$
about a given {\em alternative} $X\in\Linfty$. In combination, each expert opinion leads to an expected utility evaluation, the set of all such evaluations being $\{\E_\P[u(X)]\mid\P\in\CP\}$.
Faced with the problem of how to process this set, the DM aggregates it with a weighted average. Special cases of such weighted averages are integrating $u(X)$ with respect to a measure as in \eqref{eq:measure}, or the expression $\int_{\kappa_1}^{\kappa_2}\E_{\P^\kappa}[u(X)]\pi(d\kappa)$ in \eqref{eq:amarante} with weighting probability $\pi$.
Such aggregation procedures can, e.g., underpin risk measurement procedures; cf.\ \cite[Section 5]{Amarante}.

Our discussion in the present case study raises an important issue. {\em A priori}, one cannot distinguish whether expected utility with respect to a probability measure as in \eqref{eq:measure} is an {\em aggregation} of the opinions of other experts, or the evaluation of the opinion of a single expert.
The two perspectives can be distinguished though if the demand that expert opinions be supported is imposed. 
It is not too far-fetched that relevant expertise leads to a set of relevant scenarios as collected by an order support. 
While the evaluation $X\mapsto\E_\P[u(X)]$ for a single expert is then order continuous and therefore regular, the aggregation functional may unsurprisingly lose this property.

\medskip

\section{Brannath-Schachermayer Bipolar Theorem under uncertainty}\label{sec:BS}

Bipolar theorems play an important role in utility optimisation, see the seminal paper by \cite{Kramkov} under dominated uncertainty and the recent approach in \cite{BartlKupper2021} for non-dominated uncertainty. In this section we classify robust variants of the {\em Brannath-Schachermayer Bipolar Theorem} \cite[Theorem 1.3]{Brannath}. Its application to utility maximization in nondominated models is part of ongoing research. 

Consider the space ${L^0_{\mbf c}}$ as well as a convex and {\em solid} set $\emptyset\neq\CC\subset {L^0_{\mbf c}}_+$. $\CC$ being solid means that $X\in\CC$ and $0\peq Y\peq X$ together imply $Y\in\CC$. In a dual approach to utility maximization $\CC$ would equal the set nonnegative investment opportunities which may be superhedged at the cost of one unit of currency. 
The one-sided polar of $\CC$ is
\[\CC^\diamond:=\big\{\mu\in\ca_+\big|\,\forall\,X\in\CC:~\int X\,d\mu\le 1\big\}.\]
Classically, $\CC^\diamond$ corresponds to the set of arbitrage-free pricing rules in the context of utility maximisation.
Now consider the following properties of $\CC$:
\begin{itemize}
\item[\namedlabel{BS1}{\tn{\textbf{(BS1)}}}] $\CC$ is order closed; cf.\ Appendix~\ref{sec:prelim}. 
\item[\namedlabel{BS2}{\tn{\textbf{(BS2)}}}]
$\CC=\{X\in {L^0_{\mbf c}}_+\mid \forall\,\mu\in\CC^\diamond:~\int X\,d\mu\leq 1\}$.
\item[\namedlabel{BS3}{\tn{\textbf{(BS3)}}}]$\CC=\{X\in {L^0_{\mbf c}}_+\mid \forall\,\mu\in\CC^\diamond\cap\sca:~\int X\,d\mu\leq 1\}$.
\end{itemize}
If $\CP\approx\P^*$ for a probability measure $\P^*\in\Delta(\CF)$, the Brannath-Schachermayer Bipolar Theorem \cite[Theorem 1.3]{Brannath} states that a convex and solid set $\CC\subset\Lzeroplus$ satisfies \ref{BS2} if and only if it is closed with respect to convergence in probability under $\P^*$. The following lemma is straightforward to prove though.
\begin{lemma}
For $\P^*\in\Delta(\CF)$ and a solid subset $\CC\subset L^0_{\P^*+}$, the following are equivalent:
\begin{itemize}
    \item[(1)]$\CC$ is closed with respect to convergence in probability under $\P^*$.
    \item[(2)]$\CC$ is order closed, i.e., satisfies \ref{BS1}.
\end{itemize}
\end{lemma}
Following the formulation of \cite[Theorem 14]{surplus}, equivalences between \ref{BS1} and \ref{BS2} or \ref{BS3} can therefore be rightly called robust variants of the Brannath-Schachermayer Bipolar Theorem, the focus of the present section. \cite[Theorem 14]{surplus} to our knowledge provides the first such robust extension of the Brannath-Schachermayer Bipolar Theorem under the assumption that $\ca^*=\Linfty$; cf.\ \cite[p.\ 1361]{surplus}.
 Also, \ref{BS3} \textit{a priori} implies \ref{BS2} because of the inclusion $\sca\subset\ca$.

\begin{theorem}\label{thm:BS}
Let $\CP\subset\Delta(\CF)$ be nonempty. 
\begin{itemize}
\item[\tn{(1)}]The following are equivalent:
\begin{itemize}
\item[\tn{(i)}]$\CP$ is of class \tn{(S)}. 
\item[\tn{(ii)}]For all convex and solid sets $\emptyset\neq\CC\subset {L^0_{\mbf c}}_+$, \ref{BS1} and \ref{BS3} are equivalent.
\end{itemize}
\item[\tn{(2)}]The following are equivalent:
\begin{itemize}
\item[\tn{(i)}]$\sca=\ca$.
\item[\tn{(ii)}]For all convex and solid sets $\emptyset\neq\CC\subset {L^0_{\mbf c}}_+$, \ref{BS1} and \ref{BS2} are equivalent.
\end{itemize}
\end{itemize}
\end{theorem}

For the sake of transparency, we single out one implication in the following lemma.

\begin{lemma}\label{lem:BS}\
Suppose $\CP$ is of class \tn{(S)}. Then \ref{BS1} implies \ref{BS3}.
\end{lemma}
\begin{proof}
Under the class (S) property, $\sca$ separates the points of $\Linfty$ (Proposition~\ref{prop}(1)). In particular, $\tau:=|\sigma|(\Linfty,\sca)$ is an order-continuous locally convex-solid Hausdorff topology.\footnote{~For the definition of locally-convex solid topologies, see \cite[p.\ 172]{AliBurk2}. Given a vector lattice $\CX$ and an ideal $\CY\subset\CX^\sim$, the associated absolute weak topology $|\sigma|(\CX,\CY)$-topology is generated by the lattice seminorms $\CX\ni x\mapsto\langle |x|,|y|\rangle$, and the dual space of $(\CX,|\sigma|(\CX,\CY))$ is $\CY$ (\cite[Definition 2.32 \& Theorem 2.33]{AliBurk}). In particular, if $\CP$ is of class (S) so that $\langle\Linfty,\sca\rangle$ is a dual pair, the topologies $|\sigma|(\Linfty,\sca)$ and $\sigma(\Linfty,\sca)$ have the same closed convex sets.} Suppose $\emptyset\neq\CC\subset\Lzeroplus$ is order closed. The inclusion
\begin{center}$\CC\subset\{X\in {L^0_{\mbf c}}_+\mid \forall\,\mu\in\CC^\diamond\cap\sca:~\int X\,d\mu\leq 1\}$\end{center}
holds by the definition of $\CC^\diamond$. Conversely, consider the set $\mathcal D:=\CC\cap\Linfty$. $\mathcal D$ is nonempty (because for each $X\in\CC$ and $k\in\N$, $X\wedge k\ind_\Omega\in\mathcal D$ by solidity), convex, solid, and order closed. By \cite[Lemma 2]{surplus}, $\mathcal D$ is $\tau$-closed and therefore also $\sigma(\Linfty,\sca)$-closed. If we set $\sigma_{\mathcal D}(\mu):=\sup_{Y\in\mathcal D}\int Y\,d\mu$, $\mu\in\sca$, the Bipolar Theorem \cite[Theorem 5.103]{Ali} implies 
\[\mathcal D=\{X\in\Linfty\mid \forall\,\mu\in\sca:~\sigma_{\mathcal D}(\mu)\le 1~\Rightarrow~\int X\,d\mu\le 1\}.\]
Fix $\mu\in\sca$ with $\sigma_{\mathcal D}(\mu)\le 1$. Let $f$ be a version of $\tfrac{d\mu}{d|\mu|}$ and set $B:=\{f>0\}$, i.e.\ the positive part $\mu^+$ of $\mu$ satisfies $d\mu^+=f\ind_B\,d|\mu|$. Then 
\begin{equation}\label{eq:supportfn}\sigma_{\mathcal D}(\mu^+)=\sup_{X\in\mathcal D}\int Xf\ind_B\,d|\mu|=\sup_{X\in\mathcal D}\int (X\ind_B)d\mu\le\sigma_\mathcal D(\mu)\le 1.\end{equation}
Here we have used that $\mathcal D$ is solid. Now, for $Y\in\Linftyplus$, we have $\int Y\,d\mu\le 1$ for all $\mu\in\sca$ with $\sigma_{\mathcal D}(\mu)\le 1$ if and only if $\int Y\,d\mu\le 1$ for all $\mu\in\sca_+$ with $\sigma_{\mathcal D}(\mu)\le 1$. 
We have proved
\begin{equation}\label{eq:onlypos}\mathcal D=\big\{Y\in\Linftyplus \big| \forall\,\mu\in\mathcal D^\diamond\cap \sca:~\int Y\,d\mu\le 1\big\}.\end{equation}
Finally, $\CC^\diamond=\mathcal D^\diamond$ and hence $\mathcal \CC^\diamond\cap\sca=\mathcal D^\diamond\cap\sca$ follows with monotone convergence.
Using order closedness and solidity of $\CC$ in the last equality, we infer 
\begin{align*} 
\{X\in {L^0_{\mbf c}}_+\mid \forall\,\mu\in\CC^\diamond\cap\sca:~\int X\,d\mu\leq 1\}&=\{X\in {L^0_{\mbf c}}_+\mid \forall\,k\in\N\,\forall\,\mu\in\CC^\diamond\cap \sca:~\int (X\wedge k\ind_\Omega)\,d\mu\leq 1\}\\
&=\{X\in {L^0_{\mbf c}}_+\mid\forall k\in\N:~X\wedge k\ind_\Omega\in\mathcal D\}=\CC.
\end{align*}
This is \ref{BS3}.
\end{proof}

\begin{proof}[Proof of Theorem~\ref{thm:BS}]
For statement (1), let $\emptyset\neq\CC\subset\Lzeroplus$ be convex and solid. If (i) holds and $\CC$ satisfies \ref{BS1}, $\CC$ satisfies \ref{BS3}  by Lemma~\ref{lem:BS}. Conversely, suppose $(X_\alpha)_{\alpha\in I}\subset \CC$ is a net such that $X_\alpha\overset o\to X\in\Lzeroplus$; $\overset o\to$ denotes order convergence, cf.\ Appendix~\ref{sec:prelim}. For $k>0$ arbitrary, we infer $(X_\alpha\wedge k\ind_\Omega)_{\alpha\in I}\subset\CC$ and $X_\alpha\wedge k\ind_\Omega\overset{o}\to X\wedge k\ind_\Omega$. By Proposition~\ref{prop:(A1)-(A3)}(3), $\sca$ can be identified with the order-continuous dual $(\Linfty)^\sim_n$ of $\Linfty$. Hence, we obtain for all $\mu\in\CC^\diamond\cap\sca$ that 
\[\int(X\wedge k\ind_\Omega)d\mu=\lim_{\alpha\in I}\int (X_\alpha\wedge k\ind_\Omega)d\mu\le 1.\]
Taking the limit $k\uparrow\infty$ on the left-hand side, $\int X\,d\mu\le 1$ holds for all $\mu\in\CC^\diamond\cap\sca$. Hence, if $\CC$ has property \ref{BS3}, it also has property \ref{BS1}.\\
Now assume that (ii) holds. In order to show that $\CP$ is of class (S), it suffices to verify that for all $A\in\CF$ with $\mbf c(A)>0$ there is $\mu\in\sca_+$ such that $\mu(A)>0$. 
Consider the convex, solid, and order-closed set \begin{center}$\CC:=\{X\in {L^0_{\mbf c}}_+\mid X\ind_A=0\}$.\end{center}
By \ref{BS3}, $\mathcal C=\{X\in\Lzeroplus\mid \forall\,\mu\in\CC^\diamond:~\int X\,d\mu\le 1\}$. Suppose we find $0\neq \mu_0\in\CC^\diamond$. Then 
\[\sup_{t>0}t\mu_0(A^c)=\sup_{t>0}\int t\ind_{A^c}\,d\mu_0\le1,\]
which is only possible if $\mu_0(A^c)=0$. Hence, $\mu_0(A)=\mu_0(\Omega)>0$. Otherwise, if $\CC^\diamond=\{0\}$, $\CC=\Lzeroplus$ has to hold and $\mbf c(A)=0$, a case that we have excluded. 

\smallskip

We now focus on the equivalence in (2). 
Suppose first that (i) holds. As $\ca$ separates the points of $\Linfty$, Proposition~\ref{prop}(1) implies that $\CP$ is of class (S). By (1), \ref{BS1} and \ref{BS3} are equivalent properties of convex and solid sets $\emptyset\neq\CC\subset\Lzeroplus$. \ref{BS3} is in turn equivalent to~\ref{BS2} under (i).

Now assume that (ii) holds. In order to show that $\sca={ca_{\mbf c}}$, it suffices to show that each $\mu\in\ca_+$ is supported. To this end, consider a net $(X_\alpha)_{\alpha\in\,I}$ such that $0\peq X_\alpha\uparrow X\in\Linfty$ in order. Clearly, $\int X_\alpha\,d\mu\uparrow\sup_{\alpha\in I}\int X_\alpha\,d\mu=:s\in\R_+$. Define \begin{center}$\CC:=\{Y\in {L^0_{\mbf c}}_+\mid \int Y\,d\mu\le s\}$,\end{center}
a solid and convex subset of ${L^0_{\mbf c}}_+$. Moreover, one verifies that $\CC$ has property \ref{BS2}, which means that $\CC$ is order closed as well \ref{BS1}. As $(X_\alpha)_{\alpha\in I}\subset\CC$, we infer that that $X\in\CC$ and that $\int X\,d\mu=s$. This shows order continuity of $\mu$ or equivalently supportedness of $\mu$ by Proposition~\ref{prop:(A1)-(A3)}(3).
\end{proof}

The following concluding corollary can be understood as a bridge to our discussion of the Fatou property of risk measures below. Recall that a subset $\emptyset\neq\CC\subset\Linfty$ is \emph{solid} if, for all $X,Y\in\Linfty$, $Y\in\CC$ and $|X|\peq |Y|$ together imply $X\in\CC$.

\begin{corollary}\label{prop:Grothendieck}
Let $\CP\subset\Delta(\CF)$ be nonempty.
\begin{itemize}
\item[\tn{(1)}]The following are equivalent:
\begin{itemize}
\item[\tn{(i)}]$\CP$ is of class \tn{(S)}. 
\item[\tn{(ii)}]A convex and solid set $\emptyset\neq\CC\subset\Linfty$ is order closed iff it is $\sigma(\Linfty,\sca)$-closed.
\end{itemize}
\item[\tn{(2)}]The following are equivalent:
\begin{itemize}
\item[\tn{(i)}]$\sca=\ca$. 
\item[\tn{(ii)}]A convex and solid set $\emptyset\neq\CC\subset\Linfty$ is order closed iff it is $\sigma(\Linfty,\ca)$-closed.
\end{itemize}
\end{itemize}
\end{corollary}
\begin{proof}
(1)(i) implies (1)(ii): $\CC$ is order closed iff $\mathcal D:=\{|X|\mid X\in\CC\}$ is order closed. By Theorem~\ref{thm:BS}, $\mathcal D$ is order closed if and only if 
\[\mathcal D=\{X\in \Lzeroplus\mid \forall\,\mu\in\mathcal D^\diamond\cap\sca:~\int X\,d\mu\le 1\}.\]
Since $\int |X|\,d\mu\le 1$ if and only if $\int XY\,d\mu\le 1$ holds for all $Y\in\Linfty$ with $\|Y\|_\Linfty\le 1$ we obtain
\[\CC=\{X\in \Linfty\mid |X|\in\mathcal D\}=\bigcap_{\mu\in\mathcal D^\diamond}\bigcap_{Y\in\Linfty:~\|Y\|_\Linfty\le 1}\{X\in \Linfty\mid \smallint XY\,d\mu\le 1\}.\]
The latter set is $\sigma(\Linfty,\sca)$-closed. Conversely, every $\sigma(\Linfty,\sca)$-closed convex and solid set is order closed because $\sca$ corresponds to the order continuous dual of $\Linfty$.\\
(1)(ii) implies (1)(i): Suppose $A\in\CF$ satisfies $\mu(A)=0$ for all $\mu\in\sca_+$. In particular, for each $\mu\in\sca$ and each $s\in\R$, 
\[\int s\ind_A\,d\mu=\int s\ind_A\tfrac{d\mu}{d|\mu|}d|\mu|=0.\]
We have to show that $\mbf c(A)=0$. To this end, consider the convex, solid, and order-closed set 
\[\CC:=\{X\in\Linfty\mid |X|\ind_{A}=0\}.\]
By (1)(ii) $\CC$ is $\sigma(\Linfty,\sca)$-closed. 
The Separating Hyperplane Theorem allows to represent $\CC$ as 
\begin{center}$\{X\in \Linfty\mid\forall\,\mu\in\sca:~\sigma_\CC(\mu)<\infty\,\implies\,\int X\,d\mu\le\sigma_\CC(\mu)\}$,\end{center}
where $\sigma_\CC(\mu):=\sup_{X\in\CC}\int X\,d\mu\ge 0$. Thus we infer that 
$\R\cdot\ind_A\subset\CC$. Recalling the definition of $\CC$, $\mbf c(A)=0$ follows. This entails that $\CP$ is of class (S).

\smallskip

(2)(i) implies (2)(ii): If (i) holds, $\CP$ is again of class (S). By (1), a convex and solid set $\emptyset\neq\CC\subset\Linfty$ is order closed iff it is $\sigma(\Linfty,\sca)$-closed. The assertion follows because $\sigma(\Linfty,\sca)=\sigma(\Linfty,\ca)$.\\
(2)(ii) implies (2)(i): That each $\mu\in\ca$ is supported given that assertion (2)(ii) holds is derived like in the proof of Theorem~\ref{thm:BS}. 
\end{proof}

\begin{remark}
A remarkable model-free bipolar theorem in the spirit of Brannath-Schachermayer has recently been shown and applied in~\cite{BartlBipolar}.
It shares certain traits with the results of this section (the focus on the positive cone of a function space, solidity and convexity of the sets in question), but there are also substantial differences.
Firstly, the functions in question may attain the value $+\infty$. The authors therefore have in mind the ambient set of extended real-valued measurable functions that are bounded below. 
This is not a real vector lattice.
Secondly, the underlying measurable space has a topological structure. Thirdly, while the concept of ``$\liminf$-closedness" in \cite{BartlBipolar} is seemingly related to sequential order closedness (i.e., a set contains every limit of its order-convergent sequences; see also \cite[Remark 6]{BartlBipolar}), a full comparison to \ref{BS1} is not possible because of the lacking vector lattice structure.
\end{remark}

\medskip

\section{Superhedging, aggregation, and the Fatou property under uncertainty}\label{sec:weak*}

The present section is devoted to issues around superhedging, aggregation, and various aspects of the Fatou property of convex risk measures on $\Linfty$. All those topics are closely related, and the central question whose answer turns out to be crucial for dealing with the mentioned topics is under which condition $\Linfty$ carries a weak* topology.

\subsection{When does $\Linfty$ carry a weak* topology?}\label{sec:weak:star}

First and foremost, we need to recall that a vector lattice $(\CX,\peq)$ is {\em Dedekind complete} if every nonempty $\CC\subset\CX$ which has an upper bound $y\in\CX$, i.e.\ $x\peq y$ holds for all $x\in\CC$, has a supremum, i.e.\ a least upper bound denoted by $\sup\CC$. If $\CX$ also has the \textit{countable sup property}, that is, for every set $\CC\subset\CX$ whose supremum exists, there is a countable subset $\mathcal D\subset\CC$ such that $\sup\mathcal D=\sup\CC$, $\CX$ is said to be \textit{super Dedekind} complete.
For a moment consider a single probability measure $\P^*$ on $(\Omega,\CF)$. It is well known that $L^\infty_{\P^*}$ is super Dedekind complete, see e.g.\ \cite[Theorem 1.2.10]{Torgersen}. This fact is often used when dealing with (essential) suprema as it allows to realise such a supremum along a monotone sequence provided the underlying set of random variables is (upwards) directed; e.g., see the proofs of \cite[Theorems 9.9, 9.22, 11.2]{FS}.
More generally, by \cite[Lemma 8]{surplus}, Dedekind completeness of $\Linfty$ is equivalent to Dedekind completeness of $({L^0_{\mbf c}},\peq)$, which in turn is equivalent to Dedekind completeness of any ideal $\CX\subset {L^0_{\mbf c}}$ with the property $\Linfty\subset\CX$. The same equivalences hold for super Dedekind completeness. 

\begin{example}\label{ex:Dedekind:complete}
\begin{enumerate}
    \item Consider the lattice homomorphism $J\colon \Linfty\to\CY\subset\prod_{\Q\in\CQ}L^\infty_\Q$ defined in Section~\ref{sec:product}, where $\CP$ is of class (S) and $\CQ\approx\CP$ is a disjoint supported alternative. Using super Dedekind completeness of $L^\infty_\Q$, $\Q\in\CQ$, $\CY$ is shown to be Dedekind complete. Hence $\Linfty$ is also Dedekind complete if $J$ is onto and $\Linfty$ therefore lattice isomorphic to a product space. Otherwise, we can invoke \cite[Theorem 1.41]{AliBurk} to observe that $\CY$ is the so-called {\em Dedekind completion} of $\Linfty$; cf.\ Appendix~\ref{sec:prelim}. 
   \item Consider the context of volatility uncertainty introduced in Section~\ref{sec:voluncertainty}. In light of our discussion in Section~\ref{sec:aggregation:appl}, \cite[Section 5]{STZ} shows that $\Linfty$ is Dedekind complete if the underlying $\sigma$-algebra is sufficiently enriched. This procedure is also commented on in the additional Appendix~\ref{sec:enlargements}.
\end{enumerate}
\end{example}

As a first step, we prove a result of potential independent interest which we could not find in the literature. Note that the primal space $\CY$ in the following Proposition~\ref{prop:onlyDC} is \textit{not} assumed to be a lattice. 
The notion of archimedeanity is introduced in Appendix~\ref{sec:prelim}.

\begin{proposition}\label{prop:onlyDC}
Suppose $(\CY,\Norm_{\CY})$ is a normed space whose dual space is given by an Archimedean Banach lattice 
$(\CX,\peq,\Norm_\CX)$. If $\CX_+$ is weak* closed, then $\CX$ is Dedekind complete.
\end{proposition}
\begin{proof}
We can always assume $\CY\subset \CX^*=\CX^\sim$ using the canonical isometry. In particular, we may view $\CY$ as a subset of a vector lattice. 
As $\CX$ is Archimedean, it admits a Dedekind completion $(\CX^\delta,\tle)$, and there is an injective lattice homomorphism $J\colon\CX\to\CX^\delta$ such that the image $J(\CX)$ is an order dense and majorising vector sublattice of $\CX^\delta$; cf.\ \cite[Theorem 1.41]{AliBurk}. It can be verified that the function
\[\Norm_{\CX^\delta}\colon{\CX^\delta}\to\R_+,\quad \ell\mapsto\inf\{\|X\|_\CX\mid X\in\CX,\,|\ell|\tle J(X)\},\]
is a lattice norm on $({\CX^\delta},\tle)$.\\
Now, for all $\ell\in{\CX^\delta}$ we would like to construct a linear functional $T(\ell)\in\CY^*=\CX$. To this end, we first consider  
\[T(\ell)(\mu):=\inf_{X\in\CX:\,\ell\tle J(X)}\mu(X),\quad \ell\in\CX^\delta_+,\,\mu\in\CX^*_+,\]
and note that $T(\ell)(\mu)\in\R_+$ due to $J(\CX)$ being majorising in ${\CX^\delta}$. Moreover, this map is positively homogeneous in $\mu$. Setting $\CC_\ell:=\{X\in\CX\mid \ell\tle J(X)\}$, we have for $\mu,\nu\in\CX^*_+$ that 
\[T(\ell)(\mu+\nu)=\inf_{X\in\CC_\ell}\mu(X)+\nu(X)\ge \inf_{X,Y\in\CC_\ell}\mu(X)+\nu(Y)=T(\ell)(\mu)+T(\ell)(\nu).\]
Furthermore,
\[T(\ell)(\mu)+T(\ell)(\nu)\ge \inf_{X,Y\in\CC_\ell}(\mu+\nu)(X\wedge Y)\ge \inf_{Z\in\CC_\ell}(\mu+\nu)(Z)=T(\ell)(\mu+\nu).\]
In conclusion, $T(\ell)(\mu+\nu)=T(\ell)(\mu)+T(\ell)(\nu)$. 
Observe also that 
\begin{equation}\label{eq:bound1}|T(\ell)(\mu)|=\inf_{X\in\CC_\ell}\mu(X)\le \inf_{X\in\CC_\ell}\|X\|_\CX\cdot \|\mu\|_{\CX^*}=\|\ell\|_{\CX^\delta}\|\mu\|_{\CX^*},\quad \ell\in\CX^\delta_+,\,\mu\in\CX^*_+.\end{equation}
Now, for $\ell\in\CX^\delta_+$ and $\mu\in\CX^*$, we set $T(\ell)(\mu)=T(\ell)(\mu_1)-T(\ell)(\mu_2)$ for arbitrary $\mu_1,\mu_2\in\CX^*_+$ with $\mu=\mu_1-\mu_2$. Using the additivity of $T(\ell)$ on $\CX^*_+$ for $\ell\in\CX^\delta_+$, one proves that $T(\ell)$ is well defined and linear on $\CX^*$. 
At last, considering $\ell\in{\CX^\delta}$, we define a linear functional by $T(\ell)(\mu):=T(\ell^+)(\mu)-T(\ell^-)(\mu)$, $\mu\in\CX^*$, and obtain the following generalisation of \eqref{eq:bound1}:
\[\label{eq:bound2}|T(\ell)(\mu)|\le|T(\ell^+)(\mu)|+|T(\ell^-)(\mu)|\le T(\ell^+)(|\mu|)+T(\ell^-)(|\mu|)\le \|\ell\|_{\CX^\delta} \|\mu\|_{\CX^*}.\]
In particular, $T(\ell)|_\CY$ can be identified with an element in $\CX$. \\
Next, note that the assumption that $\CX_+$ is weak* closed together with the Hahn-Banach Separation Theorem implies the existence of a set $\mathcal Z\subset\CY$ such that 
\begin{equation}\label{eq:cone}\CX_+=\bigcap_{\mu\in\mathcal Z}\{X\in\CX\mid\mu(X)\ge 0\}.\end{equation}
In particular, $\mathcal Z\subset\CX^*_+$ has to hold and $X\succeq Y$ is equivalent to $\mu(X)\ge \mu(Y)$ for all $\mu\in\mathcal Z$.\\
For Dedekind completeness of $\CX$, it suffices to show that a nondecreasing and order bounded net $(X_\alpha)_{\alpha\in I}\subset\CX$ has a supremum in $\CX$. Without loss of generality, we may assume $(X_\alpha)_{\alpha\in I}\subset\CX_+$. As $(J(X_\alpha))_{\alpha\in I}$ is also nondecreasing and order bounded, we can set  $\ell^*:=\sup_{\alpha\in I}J(X_\alpha)$ in ${\CX^\delta}$. $T(\ell^*)|_\CY$ defines an element of $\CY^*$. 
As $\CY^*=\CX$ by assumption, there exists $X^*\in\CX$ such that
\begin{center}$\forall\,\mu\in\CY:~T(\ell^*)(\mu)=\mu(X^*)$.\end{center}
Also, for all $\mu\in\mathcal Z$, all $\alpha\in I$, and all $Y\in\CC_{\ell^*}$, 
\[\mu(X_\alpha)\le\mu(Y).\]
Taking the infimum over all $Y\in\CC_{\ell^*}$ on the right-hand side and rearranging the inequality, we infer
\[\mu(X^*-X_\alpha)=T(\ell^*)(\mu)-\mu(X_\alpha)\ge 0.\]
By \eqref{eq:cone}, $X^*$ is an upper bound of $\{X_\alpha\mid\alpha\in I\}$. If $Y\in\CX$ is an upper bound of $\{X_\alpha\mid\alpha\in I\}$, $\ell^*\tle J(Y)$ has to hold in ${\CX^\delta}$. We infer that $\mu(X^*)=T(\ell^*)(\mu)\le\mu(Y)$ for all $\mu\in\mathcal Z$. Arguing as before, $X^*\peq Y$ follows, which concludes the proof that $X^*=\sup_{\alpha\in I}X_\alpha$ in $\CX$.
\end{proof}

The following theorem establishes the relation between $\Linfty$ carrying a weak* topology and the conjunction of $\CP$ being of class \tn{(S)} and $\Linfty$ being Dedekind complete.

\begin{theorem}\label{thm:perfect}
Let $\CP$ be a set of probability measures on $(\Omega,\CF)$. Then the following are equivalent:
\begin{itemize}
\item[\tn{(1)}]$\CP$ is of class \tn{(S)} and $\Linfty$ is Dedekind complete. 
\item[\tn{(2)}]$\Linfty$ is the dual space of $\sca$.
\item[\tn{(3)}]$\Linfty$ is the dual space of a normed vector lattice.  
\item[\tn{(4)}]$\Linfty$ is the dual space of a Banach lattice.
\item[\tn{(5)}]$\Linfty$ is perfect, i.e.\ $((\Linfty)^\sim_n)^\sim_n=\Linfty$ via the embedding \eqref{eq:embedding}.
\item[\tn{(6)}]$\Linfty$ carries a locally convex-solid order-continuous Hausdorff topology and is Dedekind complete.
\end{itemize}
\end{theorem}

For the proof of Theorem~\ref{thm:perfect} we will need to introduce the following notions:

\begin{definition}\label{def:localisable}
Let $(\Omega,\CF)$ be a measurable space.
\begin{enumerate}[(1)]
\item A semi-finite\footnote{~A measure $\mu\colon\CF\to[0,\infty]$ is semi-finite if, for all $B\in\CF$ with $\mu(B)=\infty$, there is $A\in\CF$ such that $A\subset B$ and $0<\mu(A)<\infty$.} measure $\mu\colon\CF\to[0,\infty]$ is \emph{localisable} if the Boolean algebra $\{\ind_A\mid A\in\CF\}\subset L^\infty_\mu$ is Dedekind complete: for each nonempty set $\CA\subset \CF$ there is $B\in\CF$ such that $\mu(\ind_A>\ind_B)=0$ for all $A\in\CA$, and $\mu(\ind_B> \ind_C)=0$ holds for every event $C\in\CF$ with that property. 
\item A nonempty set $\CP$ of probability measures on $(\Omega,\CF)$ is \emph{weakly dominated} if it is majorised by a localisable measure $\mu$. 
\end{enumerate}
\end{definition}

\begin{proof}[Proof of Theorem~\ref{thm:perfect}]
(1) implies (2): By \cite[363M, Theorem]{Fremlin3}, Dedekind completeness of $\Linfty$ is equivalent to Dedekind completeness of the Boolean algebra $\{\ind_A\mid A\in\CF\}\subset\Linfty$.
If $\CP$ is of class (S), Lemma~\ref{lem:Luschgy} implies the existence of a majorised set of probability measures $\CQ\approx\CP$. The latter is weakly dominated by the equivalence of (B.6) and (B.1) in \cite[Theorem 2]{Luschgy}. By (B.7) in the same result, the localisable majorising measure can be chosen as 
$\mu=\sum_{\Q\in\CQ}\Q$ for any disjoint supported alternative $\CQ$ to $\CP$. Moreover, using (A.3) in \cite{Luschgy}, $\sca$ my be identified with $L^1_\mu$. Finally, by \cite[243G, Theorem]{Fremlin2}, $\Linfty=L^\infty_\mu=(L^1_\mu)^*=\sca^*$.

(2) implies (3): This follows from the fact that $\sca\subset{ba_{\mbf c}}$ is an ideal and thus a lattice in its own right. 

(3) implies (4): The norm completion of a normed vector lattice is a Banach lattice. Moreover, the dual space is preserved under norm completion. 

(4) implies (5): 
Let $(\CY,\tle,\Norm_\CY)$ be a Banach lattice whose dual is $(\Linfty,\peq,\Norm_\Linfty)$. Using \cite[Proposition 1.3.7]{MeyNie}, we again observe $\CY^*=\CY^\sim$. By \cite[Theorem 1.67]{AliBurk}, this entails Dedekind completeness of $\Linfty$. Moreover, as $\CY$ separates the points of $\Linfty$, it is lattice isomorphic to a sublattice of $(\Linfty)^\sim_n$ by \eqref{eq:embedding}; cf.\ \cite[p.\ 62]{AliBurk2}. Consequently, $(\Linfty)^\sim_n$ separates the points of $\Linfty$. Combine this with the monotonic completeness of $\Linfty$ to deduce that $\Linfty$ is perfect (\cite[Theorem 2.4.22]{MeyNie}).

(5) implies (6): Perfectness of a space is known to imply Dedekind completeness; cf.\ \cite[p.\ 63]{AliBurk2}.
Moreover, as $(\Linfty)^\sim_n$ separates the points of $((\Linfty)^\sim_n)^\sim_n$, the assumption implies that the convex-solid order-continuous topology $|\sigma|(\Linfty,(\Linfty)^\sim_n)$ is Hausdorff. 

(6) implies (1): Let $\tau$ be the locally convex-solid order-continuous Hausdorff topology on $\Linfty$ whose existence is claimed in (6). Then the dual space $(\Linfty,\tau)^*$---which separates the points of $\Linfty$ because of the Hausdorff property---may be identified with an ideal of $(\Linfty)^\sim_n$ by assumption and \cite[Theorem 2.22]{AliBurk}. The latter space therefore also separates the points of $\Linfty$. In conclusion, (6) implies the class (S) property of $\CP$ by Proposition~\ref{prop}(1). 
\end{proof}

\subsection{Superhedging, essential suprema, and aggregation versus class \tn{(S)} and Dedekind completeness}\label{sec:aggregation:appl}

In the following we relate the topic of Section~\ref{sec:weak:star} to superhedging, the existence of essential suprema, and eventually the aggregation of random variables. Given a suitable filtered robust model $(\Omega, (\CF_t)_{t\in \mbb T},\CF, \CP)$ ($\mbb T=\{0,\ldots, T\}$ or $\mbb T=[0,T]$) and a random endowment $X$ at terminal time $T$, a natural attempt is to make sense of the following formula: \begin{equation}\label{eq:superheding}Y_t:=\underset{\Q\in \CP}{\operatorname{ess\, sup}}\, \E_\Q[X|\CF_t],\quad  t\in \mbb T. \end{equation} 
Provided $(Y_t)_{t\in \mbb T}$ is a well-defined process, some version will typically turn out to be a well-behaved $\CP$-super\-martin\-gale which in turn can be decomposed into a type of $\CP$-martingale and a remainder term. The $\CP$-martingale allows the interpretation as the value process of some investment strategy and thus corresponds to a desired superhedge, see for instance \cite{NutzSonerSuper}. \eqref{eq:superheding} is commonly understood as there being a random variable $Y_t$ such that, for each $\P\in \CP$, \begin{equation}\label{eq:superhedging1}Y_t=\underset{\Q\in \CP(t,\P)}{\operatorname{ess\, sup}}\, \E_\Q[X|\CF_t]\quad \mbox{$\P$-a.s.},\end{equation} where $\CP(t,\P):=\{\Q\in \CP\mid\forall\,A\in\CF_t:~\Q(A)=\P(A)\}$. Note that the essential supremum on the right-hand side of \eqref{eq:superhedging1} is defined as usual under the probability $\P$. Thus the existence of $Y_t$ as in \eqref{eq:superhedging1} is a special case of aggregating random variables as described in \cite{STZ}:   
\begin{quote}
\small Since for each probability measure we have a well developed theory, for simultaneous stochastic analysis, we are naturally led to the following problem of aggregation. Given a family of random variables or stochastic processes, $X^\P$, indexed by probability measures $\P$, can one find an aggregator $X$ that satisfies $X=X^\P$, $\P$-almost surely for every probability measure $\P$? ... Once aggregation is achieved, then essentially all classical results of stochastic analysis generalize ...\quad\cite[p.\ 1854]{STZ} \end{quote}
Indeed, the feasibility of aggregating suitably consistent random variables---which motivates \cite{Cohen,STZ}---is essential in many applications of quasi-sure analysis, cf., e.g., \cite{Margarint,NutzSonerSuper, STZ11}. 
The same question is also tackled earlier in robust statistics \cite{Torgersen} following \cite[Definition 2.3]{Perlman}, its most relevant application being the equivalence of pairwise sufficiency and sufficiency (cf.\ \cite[Theorem 1.5.5]{Torgersen}). We adopt the terminology from this line of literature. 
\begin{definition}\label{def:consistent}Let $\CP\subset\Delta(\CF)$ be a nonempty set of probability measures and let $\CR\subset\Delta(\CF)\cap\ca$. 
\begin{enumerate}[(1)]
\item A family $(X^\Q)_{\Q\in\CR}\subset\Lzero$ is \emph{$\CR$-consistent} if, for all $\mf F\subset\CR$ with $|\mf F|=2$, there is $X^{\mf F}\in\Lzero$ such that, for all $\Q\in\mf F$, $\Q(X^\Q=X^{\mf F})=1$. 
\item A family $(X^\Q)_{\Q\in\CR}\subset\Lzero$ is \emph{$\mf R$-coherent} if there is $X\in\Lzero$ such that, for all $\Q\in\CR$, 
$\Q(X=X^\Q)=1$. The equivalence class $X$ is called an \emph{aggregator} of $(X^\Q)_{\Q\in\CR}\subset\Lzero$.
\end{enumerate}
\end{definition}

Clearly, $\CR$-coherence necessitates $\CR$-consistency. The important question is whether and under which condition $\CR$-consistency suffices to ensure $\CR$-coherence. It is known that this question is closely related to the subject of Section~\ref{sec:weak:star}. Indeed, for instance \cite[Lemme 4]{Luschgy2} shows that, for $\mf R\approx\CP$, the statement
\begin{itemize}
\item[\namedlabel{AGR}{\tn{\textbf{(AG)}}}]$(X^\Q)_{\Q\in\CR}\subset\Lzero$ is \emph{$\CR$-consistent} iff $(X^\Q)_{\Q\in\CR}\subset\Lzero$ is \emph{$\CR$-coherent}
\end{itemize}
is equivalent to $\Linfty$ being the dual space of band$(\CR)$, the band generated by $\CR$ in $\ca$. In the latter case, property (3) in Theorem~\ref{thm:perfect} is satisfied. This shows that the aggregation property \ref{AGR} is another equivalent to (1)--(6)  in Theorem~\ref{thm:perfect}.
In particular, aggregation and thus the existence of an essential supremum and of superhedging strategies following \cite{STZ} is equivalent to $\CP$ being of class \tn{(S)} and $\Linfty$ being Dedekind complete. 

\smallskip

\subsection{Super Dedekind completeness and the Kreps-Yan property}\label{sec:dominated}

Recall the definition of super Dedekind completeness in Section~\ref{sec:weak:star}. As mentioned there, super Dedekind completeness is an important feature used in many proofs of results under dominated uncertainty. We know that Dedekind completeness of $\Linfty$ is still possible when $\CP$ is nondominated (see Example~\ref{ex:Dedekind:complete}), and in fact a crucial property in view of superhedging and aggregation of random variables (see Section~\ref{sec:aggregation:appl}). Thus a natural question is whether we might also save the stronger property of super Dedekind completeness of $\Linfty$ for nondominated $\CP$. However, Proposition~\ref{thm:countsup} below shows that this is impossible. Indeed, $\Linfty$ is super Dedekind complete if and only if $\CP\approx\P^*$ for a single probability measure $\P^*$, and hence $\Linfty=L^\infty_{\P^*}$. Moreover, we will study the close relation between the Kreps-Yan property under uncertainty and  super Dedekind completeness of $\Linfty$.

\begin{proposition}\label{thm:countsup}
For a nonempty set $\CP\subset\Delta(\CF)$, the following are equivalent:
\begin{itemize}
\item[\tn{(1)}] $\Linfty$ is super Dedekind complete.
\item[\tn{(2)}] $\CP$ is dominated.
\item[\tn{(3)}] Each measure $\mu\in\ca$ is supported, i.e.\ $\sca=\ca$, and $\Linfty$ has the countable sup property.
\item[\tn{(4)}] There is a strictly positive linear functional $\xi\colon\Linfty\to\R$. 
\end{itemize}
\end{proposition}

Note that Proposition~\ref{thm:countsup} is a variation of and draws from~\cite[Theorem 5]{Luschgy}. We also refer to the references quoted there. 

\begin{proof}
(1) implies (2): Suppose $\Linfty$ is super Dedekind complete. Let $\CA\subset\CF$ be a family of events such that $\ind_A\neq 0$ holds for all $A\in\CA$, and $\ind_A\wedge\ind_B=0$ holds in $\Linfty$ whenever $A,B\in\CA$ satisfy $A\neq B$. By super Dedekind completeness of $\Linfty$, $\sup_{A\in\CA}\ind_A$ exists in $\Linfty$. Moreover, we may select a countable subset $\mathcal B\subset\CA$ such that $\sup_{A\in\CA}\ind_A=\sup_{B\in\CB}\ind_B$. Assume for contradiction that we can find $A^*\in\CA\backslash \CB$. Then 
\[\ind_{A^*}=\ind_{A^*}\wedge\sup_{A\in\CA}\ind_A=\ind_{A^*}\wedge\sup_{B\in\CB}\ind_B=\sup_{B\in\CB}\ind_{A^*}\wedge\ind_B=0,\]
where we have used \cite[Lemma 1.5]{AliBurk} in the penultimate equality. 
This contradicts the choice of $\CA$. As such, the Boolean algebra $\{\ind_A\mid A\in\CF\}\subset\Linfty$ is of countable type; see \cite[p.\ 187]{Luschgy}. By \cite[Theorem 5]{Luschgy}, $\CP$ is dominated. 

(2) implies (3): Applying a classical exhaustion argument to the dominated set $\CP$, the dominating measure $\P^*$ can be chosen as a countable convex combination of elements in $\CP$  which satisfies $\P^*\approx\CP$. $\Linfty$ is consequently lattice isomorphic to $L^\infty_{\P^*}$, which is super Dedekind complete. This also gives the countable sup property in (3).
The identity $\sca=\ca$ follows from the implication (D.1) $\Rightarrow$ (D.8) in \cite[Theorem 5]{Luschgy}. 

(3) implies (4): For each $\P\in\CP$ consider $\ind_{S(\P)}\in\Linfty$. By Lemma~\ref{lem:supremum}, $\sup_{\P\in\CP}\ind_{S(\P)}=\ind_\Omega$. By the countable sup property, there is a sequence $(\P_n)_{n\in\N}\subset\CP$ such that $\ind_\Omega=\sup_{n\in\N}\ind_{S(\P_n)}$. Using \cite[Lemma 1.5]{AliBurk} in the second equality, we get for all $N\in\CF$ that
\begin{equation}\label{eq:null}\ind_N=\ind_{N}\wedge \ind_\Omega=\sup_{n\in\N}\ind_{N}\wedge\ind_{S(\P_n)}=\sup_{n\in\N}\ind_{N\cap S(\P_n)}.\end{equation}
If $0=\P_n(N)=\P_n(N\cap S(\P_n))$ holds for some $n\in\N$, then $\ind_{N\cap S(\P_n)}=0$ follows from the definition of a support. By \eqref{eq:null}, $\sup_{n\in\N}\P_n(N)=0$ therefore implies $\ind_N=0$.
Setting 
\[\sum_{n=1}^\infty2^{-n}\P_n=:\P^*,\]
the linear functional $\E_{\P^*}[\cdot]$ is strictly positive on $\Linfty$. This is (4).

(4) implies (1):  As $\Linfty$ is $\sigma$-Dedekind complete (each countable order bounded subset $\mathcal D\subset\Linfty$ has a supremum) and a strictly positive linear functional with values in $\R$ exists, \cite[Lemma A.3]{Nendel} shows super Dedekind completeness. 
\end{proof}

We will now apply Proposition~\ref{thm:countsup} to the \textit{Kreps-Yan property}, see \cite[Definition 4.7]{Platen} or \cite{Napp, Rokhlin2}.  For a detailed discussion of the Kreps-Yan Theorem we refer to \cite[Chapter 5]{DS}.

\begin{definition}
Let $\CX$ be a vector lattice with positive cone $\CX_+$ and let $\tau$ be a locally convex-solid Hausdorff topology on $\CX$. $(\CX,\tau)$ has the \emph{Kreps-Yan property} if for all convex cones $\CC$ with the properties $-\CX_+\subset\CC$ and $\overline\CC\cap\CX_+=\{0\}$ for the $\tau$-closure $\overline{\CC}$ of $\CC$, there is a strictly positive functional $\xi\in(\CX,\tau)^*$ such that $\xi(x)\le 0$ for all $x\in\CC$.
\end{definition}

We recall that the Kreps-Yan property is closely related to the existence of state-price deflators, see for instance \cite{DS, Napp}. In fact, consider a financial market model given by a space of investment opportunities $\CX$ and the convex cone $ \CC\subset \CX$ of zero-cost investments. It is natural to assume that $-\CX_+\subset \CC$. A typical version of the FTAP would state that the absence of free lunches (no arbitrage), that is $\overline\CC \cap \CX_+=\{0\}$, is equivalent to the existence of a state price deflator, that is a strictly positive functional $\xi\in(\CX,\tau)^*$ such that $\xi(x)\le 0$ for all $x\in\CC$. Corollaries~\ref{cor:KrepsYan} and \ref{cor:KrepsYan2} below show that such theorems necessitate dominated uncertainty simply because the mentioned type of state price deflators does not exist when $\CP$ is not dominated.

\begin{corollary}\label{cor:KrepsYan}
Let $\Linfty\subset\CX\subset\Lzero$ be an ideal. Let $\tau$ be a locally convex-solid Hausdorff topology on $\CX$. If $(\CX,\tau)$ has the Kreps-Yan property, then $\CP$ is dominated. 
\end{corollary}
\begin{proof}
The positive cone $\CX_+$ of $\CX$ is closed with respect to the locally convex-solid topology $\tau$; cf.\ \cite[Theorem 2.21(b)]{AliBurk2}.
Applying the Kreps-Yan property to the convex cone $\CC:=-\CX_+$, we infer that $(\CX,\peq,\tau)$ admits a continuous strictly positive functional $\xi$. $\xi|_{\Linfty}$ is strictly positive as well, whence dominatedness of $\CP$ follows with Proposition~\ref{thm:countsup}. 
\end{proof}

\begin{corollary}\label{cor:KrepsYan2}
Let $(\CX,\peq,\Norm_\CX)$ be a Banach lattice such that  $\Linfty\subset\CX\subset {L^0_{\mbf c}}$ is an ideal.
Then the following are equivalent:
\begin{itemize}
\item[\tn{(1)}]$\CP$ is dominated. 
\item[\tn{(2)}]$\CX$ has the Kreps-Yan property.
\end{itemize}
\end{corollary}
\begin{proof}
(1) implies (2): As in the proof of Proposition~\ref{thm:countsup} we can choose a probability measure $\w\P\approx\CP$ which satisfies $\Linfty=L^\infty_{\w\P}\subset\CX\subset L^0_{\w\P}={L^0_{\mbf c}}$ and such that the $\CP$-q.s.\ order agrees with the $\w\P$-a.s.\ order. 
Consider the $\w\P$-completion $(\Omega,\CF(\w\P))$ of $(\Omega,\CF)$ as defined in \eqref{eq:defenlarge}. The $\sigma$-algebra $\CF(\w\P)$ collects all $B\subset\Omega$ with the property that, for appropriate $A,N\in\CF$, $A\triangle B\subset N$ and $\w\P(N)=0$. In particular, $\w\P$ extends uniquely to a probability measure $\w\P^\sharp$ on $\CF(\w\P)$. 
One can show that $L^0_{\w\P}$ over $(\Omega,\CF)$ and $L^0_{\w\P^\sharp}$ over $(\Omega,\CF(\w\P))$ are lattice isomorphic. Hence, we can assume without loss of generality that the underlying $\sigma$-algebra $\CF$ is $\w\P$-complete (i.e.\ $\CF=\CF(\w\P)$), which qualifies $(\CX,\peq,\Norm_\CX)$ to be a Banach ideal space in the sense of \cite{Rokhlin2}. \cite[Theorem 1]{Rokhlin2} shows that (2) holds. 

(2) implies (1): This is Corollary~\ref{cor:KrepsYan}. 
\end{proof}

\begin{remark}
In \cite[Section 2]{Napp}, it is shown that the Banach lattice 
\[(\ell^1(\R_+),\le,\Norm_{\ell^1(\R_+)}),\]
where $\ell^1(\R_+)$ comprises all functions $f\colon\R_+\to\R$ satisfying 
\begin{center}$\|f\|_{\ell^1(\R_+)}:=\sup_{F\subset\R_+\tn{ finite}}\sum_{\om\in F}|f(\om)|<\infty$,\end{center}
and $\le$ denotes the pointwise order, 
does not have the Kreps-Yan property. The linear functional 
\[\xi\colon\begin{array}{l}\ell^1(\R_+)\to\R,\\
f\mapsto\sum_{\om\in\R_+}f(\om),\end{array}\]
is strictly positive though. In light of  Corollary~\ref{cor:KrepsYan2}, the decisive feature of this example is that there is no $\sigma$-algebra $\CF$ and a set $\CP\subset\Delta(\CF)$ such that the associated upper probability $\mbf c$ satisfies $\Linfty\subset\CX\subset {L^0_{\mbf c}}$. 
\end{remark}

\subsection{The Fatou property of risk measures}\label{sec:Grothendieck}

In this subsection, we consider {\em convex risk measures} $\rho$ defined on the space $\Linfty$, i.e., functions $\rho\colon\Linfty\to\R$ with the following properties:
\begin{enumerate}[(a)]
\item convexity.
\item {\em monotonicity:} $X\peq Y$ implies $\rho(X)\le\rho(Y)$.
\item {\em cash-additivity:} For all $X\in \Linfty$ and all $m\in\R$, $\rho(X+m\ind_\Omega)=\rho(X)+m$.
\end{enumerate}
In order to justify the choice of monotonicity, we interpret the elements of $\Linfty$ as losses net of gains. 
It is well known that the defining properties of a convex risk measure $\rho\colon\Linfty\to\R$ imply that its acceptance set,
\[\CA_\rho:=\{X\in \Linfty\mid \rho(X)\le 0\},\]
is nonempty, $\Norm_\Linfty$-closed, convex, and monotone in that $X\peq Y$ and $Y\in\CA_\rho$ together imply $X\in\CA_\rho$ (slightly adapt \cite[Proposition~4.6]{FS}). Also $\rho$ may be recovered from $\CA_\rho$ via $$\rho(X)=\inf\{m\in \R\mid X-m\in \CA_\rho\},\quad X\in \Linfty.$$

One of the most widely studied topics in the theory of risk measures are so-called robust representations. Of particular interest is a representation of the convex risk measure $\rho$ as a worst-case expected net loss $\E_\Q[X]$ with respect to various conceivable probability measures $\Q$ on $(\Omega,\CF)$ whose plausibility is expressed by an adjustment $\alpha(\Q)$. That is, 
\begin{itemize}
\item[\namedlabel{G2}{\tn{\textbf{(F1)}}}]There is a function $\alpha\colon\ca\cap\Delta(\CF)\to\R\cup\{\infty\}$ bounded from below such that $\alpha\not\equiv\infty$ and 
\[\rho(X)=\sup_{\Q\in\Delta(\CF):~\Q\ll\CP}\E_\Q[X]-\alpha(\Q),\quad X\in\Linfty.\footnote{~Probability measures $\Q$ appearing in such a representation have to satisfy $\Q\ll\CP$ as $\mbf c(X=0)=1$ implies $\rho(X)=0$.}\]
\end{itemize}
The relevant question is what necessary and sufficient conditions guarantee the existence of representations as in \ref{G2}. 
If $\CP$ is dominated, and hence there is a probability measure $\P \approx\CP$, 
the answer to this question is well known. Indeed, an important result from risk measure theory (e.g., \cite[Theorem 4.33]{FS}) states that a convex risk measure $\rho$ on $\Linfty=L^\infty_\P$ satisfies \ref{G2} if and only if it has the 
{\em sequential Fatou property}:
\begin{quote}
For all sequences $(X_n)_{n\in\N}\subset L^\infty_\P$ which are bounded in norm and converge $\P$-a.s.\ to some $X\in L^\infty_\P$, $\rho(X)\le\liminf_{n\to\infty}\rho(X_n)$.
\end{quote}
This equivalence goes back at least to \cite{Delbaen} and has enjoyed various extensions to risk measures on larger model spaces than  $L^\infty_\P$ (under dominated uncertainty). \cite{GaoFatou}, for instance, proves the equivalence of the sequential Fatou property and $\ref{G2}$ when the underlying model space is an Orlicz space. \cite{Chen} shares our reverse approach and aims to characterise which conditions on the model space $\CX$ have to be satisfied so that a continuous {\em law-invariant} risk measure automatically has the Fatou property. 
Super Dedekind completeness of these model spaces for dominated $\CP$ shows that the sequential Fatou property is in fact a special case of the following general Fatou property:
\begin{itemize}
\item[\namedlabel{G1}{\tn{\textbf{(F2)}}}]$\rho$ has the Fatou property:
For all nets $(X_\alpha)_{\alpha\in I}$ and $X$ in $\Linfty$, 
\[X_\alpha\overset o\longrightarrow X\quad\implies\quad\rho(X)\le \liminf_{\alpha \in I}\rho(X_\alpha).\footnote{~\cite[Theorem 4.3]{Fatou} studies the {\em  sequential} Fatou property of quasiconvex risk measures under not necessarily dominated uncertainty and illustrates its insufficiencies. This explains why in \ref{G1} sequences are generalised by nets. % and insufficient this property is. 
We refer the interested reader to the discussion in \cite{Fatou}.}\] 
\end{itemize}
Observe that \cite{Namioka} shares with this subsection the usage of general lattice methodology. 
In particular, \cite[Proposition~24]{Namioka} shows that \ref{G1} is sufficient for obtaining a robust dual representation for convex monotone functionals on a Fr\'echet lattice $L$ if, in addition, one assumes the so-called ``C-property" for the weak topology $\sigma(L,L^\sim_n)$. The latter is debated and has been shown to fail already in dominated ambient spaces (see \cite{Cproperty}).

\medskip

This subsection is devoted to the question how \ref{G2} and \ref{G1} are related if $\CP$ is not dominated. 
A first issue arises from the observation that $\sca\neq \ca$ in general. Indeed, recall that the order-continuous dual $(\Linfty)^\sim_n$ is given by $\sca$ (Proposition~\ref{prop:(A1)-(A3)}(3)). Hence, in view of the order convergence in~\ref{G1}, it is natural to exclude non-supported probability measures in representation~\ref{G2}:
\begin{itemize}
\item[\namedlabel{G3}{\tn{\textbf{(F3)}}}]There is a function $\beta\colon\sca\cap\Delta(\CF)\to\R\cup\{\infty\}$ bounded from below such that $\beta\not\equiv\infty$ and
\[\rho(X)=\sup_{\Q\in\sca\cap\Delta(\CF)}\E_\Q[X]-\beta(\Q),\quad X\in\Linfty.\]
\end{itemize}
As shown below, $\ref{G3}\implies\ref{G1}$, whereas the general implication $\ref{G2}\implies\ref{G1}$ fails if $\ca\neq \sca$.
Also note that the---to our knowledge---only proof of $\ref{G2}\iff\ref{G1}$ in dominated frameworks combines an argument attributed to Grothendieck with the Krein-\v{S}mulian Theorem; cf.\ \cite[p.\ 1329]{Fatou}. 
Attempting to adapt this proof idea to non-dominated frameworks, we notice that  the Krein-\v{S}mulian Theorem requires that $\Linfty$ carries a weak* topology, so we are back to the discussion in Section~\ref{sec:weak:star}. In particular, this also motivates the following relaxation of \ref{G2} and \ref{G3}: Suppose $(\Linfty,\Norm_\Linfty)$ is the dual space of a normed space $(\CY,\Norm_\CY)$---which has to be isometrically isomorphic to a subspace of $ba_{\mbf c}$ equipped with the total variation norm $TV$.
\begin{itemize}
\item[\namedlabel{G4}{\tn{\textbf{(F4)}}}]There is a function $\gamma \colon\CY\to\R\cup\{\infty\}$ bounded from below such that $\gamma\not\equiv\infty$ and 
\[\rho(X)=\sup_{\mu\in\CY}\int X\,d\mu-\gamma(\mu),\quad X\in\Linfty.\]
\end{itemize}
Before we turn to the main results of this section, a remark elaborates some technical aspects.
\begin{remark}\label{rem:Fatou}
\begin{enumerate}[(1)]
\item The topology $\sigma(\Linfty,\ca)$ is always locally convex and Hausdorff. By the Fenchel-Moreau Theorem, \ref{G2} is equivalent to $\sigma(\Linfty,\ca)$-lower semicontinuity of a convex monetary risk measure. Due to cash-additivity, the latter holds if and only if the acceptance set $\CA_\rho$ is  $\sigma(\Linfty,\ca)$-closed.
\item The topology $\sigma(\Linfty,\sca)$ is  locally convex, but Hausdorff iff $\CP$ is of class (S). Without the latter, we may therefore not be able to invoke the Fenchel-Moreau Theorem. While a risk measure satisfying \ref{G3} is $\sigma(\Linfty,\sca)$-lower semicontinuous, the converse implication may not hold. 
\item \ref{G3} implies \ref{G2} and \ref{G1}. The first assertion is due to $\sca\subset\ca$. For the second, take a net $(X_\alpha)_{\alpha\in I}\subset\Linfty$ order converging to $X\in\Linfty$. Since $\sca=(\Linfty)^\sim_n$ (Proposition~\ref{prop:(A1)-(A3)}(3)), we have \[\E_\Q[X]-\beta(\Q)=\liminf_{\alpha\in I}\E_\Q[X_\alpha]-\beta(\Q)\le\liminf_{\alpha\in I}\rho(X_\alpha)\]
for all $\Q\in\sca\cap\Delta(\CF)$. Hence,
\[\rho(X)=\sup_{\Q\in\sca\cap\Delta(\CF)}\E_\Q[X]-\beta(\Q)\le\liminf_{\alpha\in I}\rho(X_\alpha).\]
\end{enumerate}
\end{remark}

The next theorem characterises when equivalences between \ref{G1}--\ref{G4} hold for general $\CP$ under the condition that the Krein-\v{S}mulian Theorem remains applicable. If a generalisation of the classical result on the relation of Fatou property and robust representations to not necessarily dominated frameworks holds ($\ref{G4}\iff\ref{G1}$), then in fact the equivalence $\ref{G1}\iff\ref{G3}$ follows. The latter therefore is the canonical generalisation and requires $\Linfty$ to have a particular structure in that $\sca^*=\Linfty$. 

\begin{theorem}\label{thm:Grothendieck}
Suppose there is a normed space $(\CY,TV)\subset(ba_{\mbf c},TV)$ whose dual space is given by $(\Linfty,\Norm_\Linfty)$.
\begin{itemize}
\item[\tn{(1)}]If, for all convex monetary risk measures $\rho$ on $\Linfty$, \ref{G1} and \ref{G4} are equivalent properties, then $\sca^*=\Linfty$. 
\item[\tn{(2)}]If $\sca^*=\Linfty$, \ref{G1} and \ref{G3} are equivalent properties  
of convex monetary risk measures $\rho$ on $\Linfty$. 
\end{itemize}
\end{theorem} 
\begin{proof}
For statement (1), consider the convex monetary risk measure 
\[\rho(X):=\inf\{m\in\R\mid X-m\preceq 0\}=\inf\{k\in\R\mid \mbf c(X>k)=0\}.\]
$\rho$ can be shown to have the Fatou property \ref{G1} and satisfies $\CA_\rho=-\Linftyplus$. By assumption, $\rho$ also satisfies \ref{G4}. In particular, it is a $\sigma(\Linfty,\CY)$-lower semicontinuous function, whence $\sigma(\Linfty,\CY)$-closedness of the acceptance set $\CA_\rho=-\Linftyplus$ follows. 
In particular, $\Linftyplus$ is also $\sigma(\Linfty,\CY)$-closed. 
By Proposition~\ref{prop:onlyDC}, $\Linfty$ is Dedekind complete.\\
Next we will show that the nontrivial set $\mathcal Z:=\{\mu\in\CY\mid\forall\,X\in\Linftyplus :~\smallint X\,d\mu\ge 0\}$
satisfies $\mathcal Z\subset\sca$. To this end, fix $\mu\in\mathcal Z$ as well as a net $(X_\alpha)_{\alpha\in I}\subset-\Linftyplus$ such that $X_\alpha\uparrow 0$ in order. 
In particular, $\iota:=\sup_{\alpha\in I}\int X_\alpha\,d\mu\le 0$. 
Set 
$$\widehat\rho(X):=\max\{\rho(X),\int X\,d\mu-\iota\},\quad X\in\Linfty,$$
a convex risk measure with property \ref{G4}. It therefore also has the Fatou property \ref{G1}. In particular, 
\[|\iota|=\widehat\rho(0)\le\liminf_{\alpha\in I}\widehat\rho(X_\alpha)\le 0.\]
Hence, $\iota=0$, and $\mu\in(\Linfty)^\sim_n=\sca$ by Proposition~\ref{prop:(A1)-(A3)}(3).\\
The class (S) property of $\CP$ now follows from $\mathcal Z\approx\CP$. To verify the latter, let $A\in\CF$ be such that $\mu(A)=0$ holds for all $\mu\in\mathcal Z$. By \eqref{eq:cone},
$\rho(s\ind_A)\le 0$ holds for all $s>0$. Equivalently, $-s\ind_A\in\Linfty_+$ holds for all $s>0$. This is only possible if $\mbf c(A)=0$, which entails $\mathcal Z\approx\CP$. At last,  Theorem~\ref{thm:perfect} yields $\sca^*=\Linfty$. 

\smallskip

Let us turn to statement (2). We have already observed in Remark~\ref{rem:Fatou} that a convex monetary risk measure $\rho$ with property \ref{G3} also has property \ref{G1}. Conversely, suppose a convex risk measure $\rho$ has the Fatou property \ref{G1}. 
We prove first that for each choice of $r>0$ the set \[\CC_r:=\{X\in\CA_\rho\mid\|X\|_\Linfty\le r\}\] is $\sigma(\Linfty,\sca)$-closed. Without loss of generality, we may assume $\CC_r\neq\emptyset$. As each $X\in\CC_r$ satisfies $X\succeq -r\ind_\Omega$ and $\CA_\rho$ is monotone, $-r\ind_\Omega=\inf\CC_r\in\CC_r$. 
Now define 
\[\mathcal B:=\{X+r\ind_\Omega\mid X\in\CC_r\}\quad\text{and}\quad\CB':=\{X\in\Linfty\mid |X|\in\CB\}.\]
Both sets are order closed and convex. Moreover, $\CB'$ is solid. Indeed, if $Z\in\CB'$ satisfies $|Z|=r\ind_\Omega+X$ for some $X\in\CC_r$, and $Y\in\Linfty$ satisfies $|Y|\peq|Z|$, we have
$$X=|Z|-r\ind_\Omega\succeq |Y|-r\ind_\Omega\succeq-r\ind_\Omega.$$
Using monotonicity of $\CA_\rho$ and the fact that $X,-r\ind_\Omega\in\CC_r$, we infer $|Y|-r\ind_\Omega\in\CC_r$ and eventually $Y\in\CB'$. As $\sca^*=\Linfty$ entails the class (S) property of $\CP$, Proposition~\ref{prop:Grothendieck}(1) implies that 
$\CB'$ is $\sigma(\Linfty,\sca)$-closed. 
As $\CB=\CB'\cap\Linftyplus$, $\CB$ is also $\sigma(\Linfty,\sca)$-closed.
Hence, the convex set $\CC_r=\CB-\{r\ind_\Omega\}$ is $\sigma(\Linfty,\sca)$-closed as well. We may now invoke the Krein-\v{S}mulian Theorem to see that $\CA_\rho$ is $\sigma(\Linfty,\sca)$-closed. Noting that, for all $k\in\R$, $\{X\in\Linfty\mid \rho(X)\le k\}=\CA_\rho+\{k\ind_\Omega\}$, this implies that $\rho$ is $\sigma(\Linfty,\sca)$-lower semicontinuous. As elaborated in Remark~\ref{rem:Fatou}, the Fenchel-Moreau Theorem yields \ref{G3}. 
\end{proof}

The following corollary returns to the starting point of our discussion, axiom~\ref{G2}. In fact, equivalence between \ref{G2} and \ref{G1} can only be expected if the fundamental condition $\ca^*=\Linfty$ from \cite{Fatou} holds.

\begin{corollary}\label{cor:Grothendieck}
Suppose $\Linfty$ is Dedekind complete. Then the following are equivalent:
\begin{itemize}
\item[\tn{(1)}]For all convex monetary risk measures $\rho$ on $\Linfty$, \ref{G2} and \ref{G1} are equivalent.
\item[\tn{(2)}]$\ca^*=\Linfty$.
\end{itemize}  
\end{corollary}
\begin{proof}
(1) implies (2): Let $0\neq \mu\in\ca_+$. Then
\[\rho(X):=\tfrac 1{\mu(\Omega)}\int X\,d\mu,\quad X\in\Linfty,\]
is a convex monetary risk measure which satisfies \ref{G2} and therefore also \ref{G1}. By Proposition~\ref{prop:(A1)-(A3)}(3), this is only possible if $\mu\in\sca$. 
As consequently $\sca=\ca$, Dedekind completeness of $\Linfty$ in conjunction with Theorem~\ref{thm:perfect} finally imply $\ca^*=\sca^*=\Linfty$. 

(2) implies (1): If $\ca^*=\Linfty$, we have $\sca=\ca$ and the equivalence of \ref{G1} and \ref{G3} for convex monetary risk measures $\rho\colon\Linfty\to\R$ (Theorem~\ref{thm:Grothendieck}(2)). As then \ref{G2} and \ref{G3} are the same, the assertion is proved. 
\end{proof}

\begin{remark}
In our investigation of the Fatou property under uncertainty we could have equivalently worked with the larger class of quasiconvex risk measures. These are generally not cash-additive. While the mathematical approach---verifying closedness properties of convex monotone sublevel sets---is the same, the dual representation of a quasiconvex risk measure is less informative and transparent. 
\end{remark}

\subsection{Concrete instances of $\Linfty$ are Dedekind complete only if $\CP$ is of class \tn{(S)}}\label{sec:logic:cor}

This concluding subsection demonstrates how nondominated robust models test the limits of {\bf ZFC}. Although there is no proof that Dedekind completeness of $\Linfty$ implies that $\CP$ is of class \tn{(S)}, any \textit{concrete} example of $\CP$ such that $\Linfty$ is Dedekind complete will turn out to be of class \tn{(S)}; cf.\ Corollary~\ref{cor:example}.

A nonempty set $\mf X$ admits a solution to Banach's measure problem if there is a diffuse probability measure $\pi\colon2^{\mf X}\to[0,1]$ on the power set $2^{\mf X}$, i.e.\ $\pi(\{\mf x\})=0$ for all $\mf x\in \mf X$. Banach's measure problem is said to have a solution if there is a nonempty set $\mf X$ which admits a solution to Banach's measure problem. The following lemma follows with \cite[Lemmas 6 \& 7, Theorem 3]{Luschgy}.

\begin{lemma}\label{lem:main1}
Suppose $\CP\subset\Delta(\CF)$ is nonempty. Then the following are equivalent: 
\begin{itemize}
\item[\tn{(1)}]${ca_{\mbf c}}^*=\Linfty$.
\item[\tn{(2)}]$\CP$ is of class \tn{(S)}, $\Linfty$ is Dedekind complete, and no disjoint supported alternative $\CQ\approx \CP$ admits a solution to Banach's measure problem. 
\end{itemize}
\end{lemma}

\begin{lemma}\label{lem:main2}Consider the following statements:
\begin{itemize}
\item[\tn{(1)}]Banach's measure problem has no solution.
\item[\tn{(2)}]$\ca^*=\Linfty$ if and only if $\Linfty$ is Dedekind complete.
\item[\tn{(3)}]$\Linfty$ is Dedekind complete only if $\CP$ is of class \tn{(S)}.
\end{itemize}
Then $\tn{(1)}\iff\tn{(2)}\implies\tn{(3)}$.
\end{lemma}
\begin{proof}
(1) implies (2): By \cite[363S, Theorem]{Fremlin3}, (1) holds if and only if every Dedekind complete vector lattice $\CX$ satisfies $\CX^\sim_c=\CX^\sim_n$. If $\Linfty$ is Dedekind complete, Proposition~\ref{prop:(A1)-(A3)} therefore implies ${ca_{\mbf c}}=(\Linfty)^\sim_c=(\Linfty)^\sim_n=\sca$. In particular,  $\CP$ is of class (S). Theorem~\ref{thm:perfect} yields $\ca^*=\Linfty$. Conversely, if $\ca^*=\Linfty$, Dedekind completeness of the latter space follows from Theorem~\ref{prop}(2).

(2) implies (1): Let $\Omega$ be an arbitrary nonempty set. Consider the measurable space $(\Omega,2^\Omega)$ and the set $\CP=\{\delta_\om\mid \om\in\Omega\}$ of probability measures which is of class (S). As $\Linfty$ agrees with the space of all bounded functions $f\colon\Omega\to\R$ and this space is Dedekind complete, (2) implies ${ca_{\mbf c}}^*=\Linfty$, and Lemma~\ref{lem:main1} implies that $\CP$ does not admit a solution to Banach's measure problem, so neither does $\Omega$.

At last, we have seen above that (1) implies the class (S) property of $\CP$. 
\end{proof}

In the subsequent corollary, we say that a property is {\em verifiable in {\bf ZFC}} if in {\bf ZFC} one can show that it holds. The terminology ``to construct an example in {\bf ZFC}" means to give an example in {\bf ZFC} in which the involved properties are verifiable. 
\begin{corollary}\label{cor:example}
It is impossible to construct an example in {\bf ZFC} of a measurable space $(\Omega,\CF)$ and a nonempty set $\CP\subset\Delta(\CF)$ such that $\Linfty$ is Dedekind complete and  additionally at least one of the following properties holds:
\begin{itemize}
\item[\tn{(1)}]$\ca^*\neq\Linfty$.
\item[\tn{(2)}]$\CP$ is not of class \tn{(S)}.
\item[\tn{(3)}]$\Linfty$ is not the dual space of a normed space $(\CY,\Norm_\CY)$. 
\end{itemize}
\end{corollary}
\begin{proof}
Suppose the construction of an example as described in (1)--(2) is possible in {\bf ZFC}. According to Lemma~\ref{lem:main2}, Banach's measure problem must have a solution.  
Similarly, if an example as in (3) exists, Theorem~\ref{thm:perfect} implies that one deals with a Dedekind complete space $\Linfty$ such that $\CP$ is not of class (S). Hence, also such an example proves that Banach's measure problem has a solution. 
Let $\kappa=|\mf X|$ be the least cardinal which admits such a solution. By \cite[Corollaries 10.7 \& 10.15]{Jech}, $\kappa$ is weakly inaccessible. As such, the construction would provide a proof of the existence of weakly inaccessible cardinals in {\bf ZFC}, and such a proof is known to be impossible; cf.\ \cite[p.\ 16]{Kanamori}. 
\end{proof}
Let us close with two important consequences of Corollary~\ref{cor:example}. Consider a concrete example of $\CP$ for which:
\begin{itemize}
\item $\CP$ being of class \tn{(S)} and $\Linfty$ being Dedekind complete are decidable in {\bf ZFC}.\footnote{~A property is decidable in {\bf ZFC} if said property or its negation are verifiable in {\bf ZFC}.} Then $\Linfty$ is Dedekind complete only if $\CP$ is of class \tn{(S)} (Corollary~\ref{cor:example}(2)).
\item $\Linfty$ being Dedekind complete and $\ca^*=\Linfty$ are decidable in {\bf ZFC}. Then $\Linfty$ is Dedekind complete only if $\ca^*=\Linfty$ (Corollary~\ref{cor:example}(1)). The latter is a crucial addendum to \cite{FatouCor}. 
\end{itemize}

\appendix

\section{Vector lattices and spaces of measures}\label{sec:prelim}

This appendix is a brief recap of the theory as presented in the monographs \cite{Ali,AliBurk,AliBurk2,MeyNie}. 

\smallskip

A tuple $(\CX,\peq)$ is a \emph{vector lattice} if $\CX$ is a real vector space and $\peq$ is a partial order on $\CX$ with the following properties: 
\begin{itemize}
\item For $x,y,z\in\CX$ and scalars $\alpha\geq 0$, $x\peq y$ implies $\alpha x+z\peq \alpha y+z$.
\item For all $x,y\in\CX$ there is a least upper bound $z:=x\vee y=\sup\{x,y\}$, the \textit{maximum} of $x$ and $y$, which satisfies $x\peq z$ and $y\peq z$ as well as $z\peq z'$ whenever $x\peq z'$ and $y\peq z'$ hold.
\end{itemize}
The existence of the absolute value $|x|=x\vee (-x)$, the positive part $x^+=x\vee 0$, the negative part $x^-=(-x)\vee 0$, and the minimum $x\wedge y:=\inf\{x,y\}=-\sup\{-x,-y\}$ follow. The \emph{positive cone} $\CX_+$ is the set of all $x\in\CX$ such that $x\succeq 0$. 
If $\CX$ is additionally carries a norm $\Norm$ which satisfies $\|x\|\leq \|y\|$ whenever $|x|\peq |y|$, $(\CX,\peq,\Norm)$ is a \emph{normed vector lattice}. A normed vector lattice $(\CX,\peq,\Norm)$ is a {\em Banach lattice} if $\Norm$ is complete. 

\smallskip

\textbf{Linear functionals:} 
If for a linear functional $\phi\colon\CX\to\R$ each set $\{\phi(z)\mid x\peq z\peq y\}\subset\R$ is bounded, $x,y\in\CX$,  $\phi$ is \emph{order bounded} and belongs to the {\em order dual} $\CX^\sim$ of $\CX$. The latter carries an order given by the positive cone $\CX^\sim_+$ comprising all order bounded linear functionals $\phi\in\CX^\sim$ satisfying $\phi(x)\geq 0$ for all $x\in\CX_+$.

A net $(x_\alpha)_{\alpha\in I}\subset\CX$ is order convergent to $x\in\CX$ if there is another net $(y_\alpha)_{\alpha\in I}$ which is \textit{decreasing} ($\alpha,\beta\in I$ and $\alpha\leq \beta$ implies $y_\beta\peq y_\alpha$), satisfies $\inf_{\alpha\in I}y_\alpha:=\inf\{y_\alpha\mid\alpha\in I\}=0$, and for all $\alpha\in I$ it holds that $0\peq |x_\alpha-x|\peq y_\alpha$. The \emph{order continuous dual} is the space $\CX^\sim_n\subset\CX^\sim$ of all order bounded linear functionals $\phi$ which are order continuous, i.e.\ $\phi$ carries an order convergent net with limit $x$ in $\CX$ to a net converging to $\phi(x)$ in $\R$. The \emph{$\sigma$-order continuous dual} is the space $\CX^\sim_c\subset\CX^\sim$ of all order bounded linear functionals $\phi$ which carry an order convergent sequence with limit $x$ in $\CX$ to a convergent sequence with limit $\phi(x)$ in $\R$. Obviously, $\CX^\sim_n\subset\CX^\sim_c$, and both spaces are vector lattices in their own right. 

Given an ideal $\CB\subset\CX^\sim$ (see below) and $x\in\CX$, the linear functional 
\begin{equation}\label{eq:embedding}\ell_x\colon\CB\ni\phi\mapsto \phi(x)\end{equation}
is order continuous on the lattice $(\CB,\peq)$. The map $\ell_\bullet\colon\CX\to\CB^\sim_n$ is a lattice homomorphism (i.e.\ it preserves the lattice structure) and is injective if and only if $\CB$ separates the points of $\CX$ (i.e.\ $\phi(x)=0$ holds for all $\phi\in\CB$ only if $x=0$). A vector lattice is \emph{perfect} if $\ell_\bullet\colon\CX\to(\CX^\sim_n)^\sim_n$ is bijective, i.e.\ $(\CX^\sim_n)^\sim_n$ may be canonically identified with $\CX$. 

\smallskip

\textbf{Vector sublattices, ideals and bands:} Given a vector lattice $\CX$, a subspace $\CY\subset\CX$ is a \emph{vector sublattice} if for every $x,y\in\CY$, the maximum $x\vee y$ computed in $\CX$ lies in $\CY$. It is \emph{order dense} in $\CX$ if for all $0\prec x\in\CX$ we can find some $y\in\CY$ such that $0\prec y\peq x$. It is \emph{majorising} if for every $x\in\CX$ there is $y\in\CY$ such that $x\peq y$. 

A vector subspace $\CB$ of $\CX$ with the property that, for all $y\in\CB$, $\{x\in\CX\mid |x|\peq |y|\}\subset\CB$, is an \emph{ideal}. Every ideal is a vector sublattice. A subset $\CC\subset\CX$ is {\em order closed} if the limit of each order convergent net $(x_\alpha)_{\alpha\in I}\subset\CC$ lies in $\CC$. An ideal $\CB\subset\CX$ is a {\em band} if it is an order closed subset of $\CX$. The \emph{disjoint complement} of an ideal $\CB$ is defined by 
\[\CB^{\tn d}:=\{x\in\CX\mid\,\forall\,y\in\CB:~|x|\wedge |y|=0\}.\]
It is always a band.
Given $\phi\in\CX^\sim$, its \emph{null ideal} is the ideal $N(\phi):=\{x\in\CX\mid |\phi|(|x|)=0\}$, and its \emph{carrier} is $C(\phi):=N(\phi)^{\tn d}$.

We call a Dedekind complete vector lattice $(\CX^\delta,\tle)$ the \emph{Dedekind completion} of $(\CX,\peq)$ if there is a linear map $J\colon\CX\to {\CX^\delta}$ such that (i) $J$ is a strictly positive lattice homomorphism ($J(|x|)=0$ for $x\in\CX$ implies $x=0$), (ii) $J(\CX)$ is an order dense and  majorising vector sublattice of ${\CX^\delta}$. $\CX$  has a Dedekind completion if and only if $\CX$ is Archimedean, i.e.\ $\tfrac 1 nx\overset o\longrightarrow 0$, $n\to\infty$, for every $x\in\CX_+$. Last, a \emph{lattice isomorphism} is a bijective lattice homomorphism.\footnote{~The term ``lattice isomorphism" is ambiguous in the literature. \cite[Definition 1.30]{AliBurk}, for instance, replaces our assumption of bijectivity by mere injectivity. By \cite[p.\ 16]{AliBurk}, however, two vector lattices are \textit{lattice isomorphic} if there is a surjective lattice isomorphism between them. It is therefore worth pointing out that the results from \cite{AliBurk} we use deal with lattice isomorphisms which are bijective.} 

\medskip

\section{Auxiliary results}\label{sec:appendix}

\begin{lemma}\label{lem:sets}
Suppose that for a set of events $\CA\subset\CF$ the supremum $\sup_{A\in\CA}\ind_A$ exists in ${L^0_{\mbf c}}$. Then there is an event $B\in\CF$ such that 
\begin{equation}\label{eq:app1}\ind_B=\sup_{A\in\CA}\ind_A.\end{equation}
Analogously, if $\inf_{A\in\CA}\ind_A$ exists in ${L^0_{\mbf c}}$, there is an event $C\in\CF$ such that 
\begin{equation}\label{eq:app2}\ind_C=\inf_{A\in\CA}\ind_A.\end{equation}
\end{lemma}
\begin{proof}
For \eqref{eq:app1}, suppose $U:=\sup_{A\in\CA}\ind_A$ exists. In particular, $U=U^+$ and $U\peq\ind_{\Omega}$ has to hold, i.e.\ $0\peq U\peq \ind_\Omega$. As for all $n\in\N$ the identity 
$\{\ind_A\mid A\in\CA\}=\{(n\ind_A)\wedge \ind_\Omega\mid A\in\CA\}$ holds, we obtain from \cite[Lemma 1.5]{AliBurk} for all $n\in\N$
\begin{align*}U&=\sup_{A\in\CA}((n\ind_A)\wedge \ind_\Omega)=nU\wedge \ind_{\Omega}.
\end{align*}
Note that $\sup_{n\in\N}(nU\wedge\ind_\Omega)=\ind_{\{u>0\}}$, where $u\in U$ is an arbitrary representative. Hence, we may set $B:=\{u>0\}$. \eqref{eq:app2} follows from \eqref{eq:app1} as $\inf_{A\in\CA}\ind_A=\ind_\Omega-\sup_{A\in\CA}\ind_{A^c}$ by \cite[Lemma 1.4]{AliBurk}. 
\end{proof}

\begin{lemma}\label{lem:supremum}
Suppose $\CP\subset\Delta(\CF)$ is of class \tn{(S)} with supported alternative $\CQ$. For all $X\in {L^0_{\mbf c}}_+$, $\sup_{\QW\in\CQ}X\ind_{S(\QW)}$ exists and is given by $X$. 
\end{lemma}
\begin{proof}
Consider the set $\CC:=\{X\ind_{S(\QW)}\mid\QW\in\CQ\}$ which is order bounded from above by $X$. Moreover, $X$ is indeed the least upper bound of $\CC$. In order to prove this, consider any upper bound $Y$. Then, for all $\Q\in\CQ$, $X\ind_{S(\Q)}\peq Y\ind_{S(\Q)}$, whence we infer $\Q(X\le Y)=1$.
As $\CP\approx\CQ$, $\mbf c(X>Y)=0$. Equivalently, $X\peq Y$. 
\end{proof}

Given a Banach lattice $(\CX,\Norm_\CX)$, where $\Linfty\subset\CX\subset\Lzero$ holds, $ca(\CX)$ denotes in the following results the set of all $\mu\in\ca$ such that each $X\in\CX$ is $|\mu|$-integrable and such that
\[\phi_\mu(X):=\int X\,d\mu,\quad X\in\CX,\]
defines a continuous linear functional on $\CX$. $sca(\CX)$ is defined in complete analogy.

\begin{proposition}\label{prop:(A1)-(A3)}
Let $\emptyset\neq\CP\subset\Delta(\CF)$ and let $\Linfty\subset\CX\subset\Lzero$ be a Banach lattice such that $\CX$ is an ideal in $\Lzero$. 
\begin{itemize}
\item[\tn{(1)}]$ca(\CX)=\CX^\sim_c$.
\item[\tn{(2)}]For $\mu\in sca(\CX)$, carrier and null ideal of $\phi_\mu$ are given by $C(\phi_\mu)=\ind_{S(\mu)}\CX$ and $N(\phi_\mu)=\ind_{S(\mu)^c}\CX$, respectively. Both are bands in $\CX$.
\item[\tn{(3)}] $sca(\CX)=\CX^\sim_n$.
\end{itemize}
\end{proposition}
\begin{proof}For statement (1), $ca(\CX)\subset \CX^\sim_c$ is clear by dominated convergence. Conversely, $\sigma$-order continuity of $\phi\in \CX^\sim_c\cap \CX^\sim_+$ together with the Daniell-Stone Theorem~\cite[Theorem 7.8.1]{Bogachev} provides a unique finite measure $\mu$ on $\CF$ such that \[\phi(X)=\int X\,d\mu,\quad X\in\CX.\] $\mu\in ca(\CX)$ follows from $\CX^\sim_c\subset\CX^\sim = \CX^*$ where $\CX^\sim = \CX^*$ is due to \cite[Proposition~1.3.7]{MeyNie}. For general $\phi\in \CX^\sim_c$ we have that $\phi=\phi^+-\phi^-$ where $\phi^+,\phi^-\in \CX^\sim_c\cap \CX^\sim_+$. 

\smallskip

Concerning statement (2), one first shows that $|\phi_\mu|=\phi_{|\mu|}$ for all $\mu\in ca(\CX)$. 
As $C(\phi_\mu)=C(\phi_{|\mu|})$, $N(\phi_\mu)= N(\phi_{|\mu|})$, and $S(\mu)=S(|\mu|)$, we can thus assume $\mu\in sca(\CX)_+$ without loss of generality. For all $X\in\CX$, $\mu(S(\mu)^c)=0$ implies 
\[\phi_\mu(|X\ind_{S(\mu)^c}|)=\int |X|\ind_{S(\mu)^c}\,d\mu=0.\]
Conversely, if $X\in\CX$ satisfies $\phi_\mu(|X|)=0$, then
\begin{align*}
0&=\phi_\mu(|X|)\geq \phi_\mu(|X|\ind_{S(\mu)})=\int|X|\ind_{S(\mu)}d\mu\geq 0.
\end{align*}
This implies $|X|\ind_{S(\mu)}=0$ in $\CX$ by property (b) in Definition~\ref{def:support}(1), and $X=X\ind_{S(\mu)^c}$. At last, 
$C(\phi_\mu)=N(\phi_\mu)^{\tn d}=(\ind_{S(\mu)^c}\CX)^{\tn d}=\ind_{S(\mu)}\CX$.
Both are indeed bands. We will only prove this for $\ind_{S(\mu)}\CX$. If $(X_\alpha)_{\alpha\in I}\subset \ind_{S(\mu)}\CX$ is a net which converges in order to $X\in\CX$, we also have $X_\alpha\ind_{S(\mu)}\to X\ind_{S(\mu)}$ in order. However, $X_\alpha\ind_{S(\mu)}=X_\alpha$, and order limits are unique. Hence, $X\ind_{S(\mu)}=X$, which means precisely that $X\in\ind_{S(\mu)}\CX$.

\smallskip

At last, we turn to statement (3). $sca(\CX)$ is a lattice (Lemma~\ref{lem:Luschgy}) and $\CX^\sim_n$ is a band in $\CX^\sim=\CX^*$ (\cite[Proposition 1.3.9]{MeyNie}). Hence, it suffices to focus on positive elements.\\
Let $\mu\in sca(\CX)_+$. In order to see that $\phi_\mu\in\CX^\sim_n$, let $(X_\alpha)_{\alpha\in I}$ be a net such that $X_\alpha\downarrow 0$. Then $X_\alpha\ind_{S(\mu)}\downarrow 0$. By (2) and \cite[Lemma 1.80]{AliBurk}, the carrier $C(\phi_\mu)=\ind_{S(\mu)}\CX$ of the functional $\phi_\mu\in\CX^\sim_c$ has the countable sup property when equipped with the $\CP$-q.s.\ order $\peq$. Thus there is a countable subnet $(\alpha_n)_{n\in\N}$ such that
$X_{\alpha_n}\ind_{S(\mu)}\downarrow 0$ in order, in particular $\mu$-almost everywhere. By monotone convergence, 
\[0\leq \inf_{\alpha}\phi_\mu(X_\alpha)= \inf_{\alpha}\phi_\mu(X_\alpha \ind_{S(\mu)})\leq \inf_{n\in\N}\phi_\mu(X_{\alpha_n} \ind_{S(\mu)})=\inf_{n\in\N}\int X_{\alpha_n} \ind_{S(\mu)}\,d\mu=0,\]
whence order continuity of $\phi_\mu$ follows.\\
For the converse inclusion, pick $\psi\in(\CX^\sim_n)_+$ and let $\mu\in ca(\CX)_+$ be such that $\psi=\phi_\mu$ (see (1)). Let $(\CX^\delta,\tle)$ be the Dedekind completion of $(\CX,\peq)$, which exists because $\CX$ is Archimedean; see Appendix~\ref{sec:prelim}. Let $J\colon\CX\to\CX^\delta$ be a lattice isomorphism such that $J(\CX)$ is order dense and majorising in $\CX^\delta$. The latter gives rise to a lattice isomorphism $\w\cdot\colon\CX^\sim_n\to(\CX^\delta)^\sim_n$ given on positive functionals $\phi\in(\CX^\sim_n)_+$ by
\begin{equation}\label{eq:ext:2}\w\phi\colon \CX^\delta\to\R,\quad Y\mapsto\sup\{\phi(X)\mid X\in\CX,\, J(X)\tle Y\};\end{equation}
cf.\  \cite[Theorem 1.84]{AliBurk} and its proof. 
Both the null ideal $N(\w\psi)$ and the carrier $C(\w\psi)$ are bands in $\CX^\delta$. By Dedekind completeness of the latter space, $\CX^\delta=N(\w\psi)\oplus C(\w\psi)$. In particular, there is $Y^*\in N(\w\psi)$ and $Z^*\in C(\w\psi)$ such that $J(\ind_\Omega)=Y^*+Z^*$.  $J(\CX)$ is order dense in $\CX^\delta$, hence one can show 
\[Z^*=\sup\{J(X)\mid X\in \CX,\,J(X)\tle Z^*\}=\sup\{J(\ind_A)\mid A\in\CF,\,J(\ind_A)\tle Z^*\}.\]
$C(\w\psi)$ has the countable sup property, hence there is an increasing sequence $(S_n)_{n\in\N}\subset\CF$ such that $J(\ind_{S_n})\uparrow Z^*$ in $\CX^\delta$. We claim that $S:=\bigcup_{n\in\N}S_n$ is a version of the support of $\mu$. Indeed, 
\[\mu(\Omega)=\w\psi(J(\ind_\Omega))=\w\psi(Y^*)+\w\psi(Z^*)=\w\psi(Z^*)=\lim_{n\to\infty}\mu(S_n)=\mu(S).\]
Moreover, for $N\in\CF$ with $\mu(N)=0$ and $N\subset S$, $J(\ind_{N\cap S_n})\in N(\w\psi)\cap C(\w\psi)$, $n\in\N$, whence $J(\ind_{N\cap S_n})=0\in\CX^\delta$ follows. As $J$ is strictly positive, $\ind_{N\cap S_n}=0\in\CX$ has to hold, and we infer $\ind_N=\sup_{n\in\N}\ind_{N\cap S_n}=0$. Definition~\ref{def:support}(1) provides $S=S(\mu)$.
\end{proof}

\begin{remark}
Proposition~\ref{prop:(A1)-(A3)}(3) is proved in the spirit of \cite{AliBurk,AliBurk2}. Another proof, which would in our opinion be less accessible for readers of this paper, can be given using results in \cite{Fremlin3}, in particular \cite[326O, Proposition]{Fremlin3}.
\end{remark}

\begin{proposition}\label{prop}
Let $\emptyset\neq\CP\subset\Delta(\CF)$ and let $\Linfty\subset\CX\subset\Lzero$ be a Banach lattice such that $\CX$ is an ideal in $\Lzero$. 
\begin{itemize}
    \item[(1)]If $sca(\CX)$ separates the points of $\CX$, then $\CP$ of class \tn{(S)}.
    \item[(2)]The converse implication holds if the supported alternative $\CQ$ can be chosen such that $\CQ\subset sca(\CX)$. 
\end{itemize}
\end{proposition}
\begin{proof}For (1), suppose $A\in\CF$ satisfies $\mbf c(A)>0$. 
Then there is $\mu\in sca(\CX)$ such that $\mu(A)=\phi_\mu(\ind_A)>0$. 
This entails that also $\mu^+(A)>0$. The set $\CQ:=\{\mu(\Omega)^{-1}\mu\mid \mu\in sca(\CX)_+\}$ therefore serves as a supported alternative to $\CP$. The proof of statement (2) is straightforward.
\end{proof}

\medskip

\section{Enlargements of $\sigma$-algebras}\label{sec:enlargements}

This paper has uncovered the necessity of  Dedekind completeness for the validity of important results from mathematical finance in the presence of uncertainty. On another note, many pertinent contributions---see the references below---weaken the notion of measurability \textit{a posteriori} by enlarging the underlying $\sigma$-algebra $\CF$ on $\Omega$. This additional appendix therefore asks under which circumstances this produces the Dedekind completion of $\Linfty$ (by additionally changing to a new set of probability measures on the larger $\sigma$-algebra). Particular emphasis will be put on the so-called universal completion.

\begin{definition}\label{def:enlargement}
Given a measurable space $(\Omega,\CF)$ and a nonempty $\CP\subset\Delta(\CF)$, an \emph{enlargement} is a tuple $(\CG,\w\CP)$, where $\CG\supset\CF$ is a $\sigma$-algebra on $\Omega$, and $\w\CP\subset\Delta(\CG)$ is such that 
\[\w\CP|_\CF:=\{\w\P|_{\CF}\mid\w\P\in\w\CP\}\approx\CP.\]
Denoting by $\mbf{\w c}$ the upper probability associated with $\w\CP$ and by $\la f\ra$ the equivalence class generated by an $\CF$-measurable random variable $f$ in $L^\infty_{\mbf{\w c}}$, we introduce the lattice homomorphism
\begin{equation}\label{eq:iota}J_\CG\colon\begin{array}{l}\Linfty\to L^\infty_{\mbf{\w c}},\\
X=[f]\mapsto\la f\ra.\end{array}\end{equation}
An enlargement $(\CG,\w\CP)$ \emph{completes} $\Linfty$ if $L^\infty_{\mbf{\w c}}$ is the Dedekind completion of $\Linfty$ and $J_\CG(\Linfty)$ is order dense and majorising in $L^\infty_{\mbf{\w c}}$. 
\end{definition}

The condition $\w\CP|_\CF\approx\CP$ ensures that $J_\CG$ is well defined and one-to-one, which is necessary if $L^\infty_{\mbf{\w c}}$ completes $\Linfty$. 
We shall introduce two relevant candidates for completing enlargements based on completing the original $\sigma$-algebra; cf.~\eqref{eq:defenlarge}.
The \textit{universal enlargement} $(\CH,\CP^\CH)$ is defined by $\CH:=\CF(\Delta(\CF))$, the universal completion of $\CF$, and the set $\CP^\CH:=\{\P^\CH\mid\P\in\CP\}$ of unique extensions of the initial probability measures $\P\in\CP$ to $(\Omega,\CH)$. The resulting upper probability on $\CH$ is denoted by $\mbf c^\CH$. Each $\mu\in ca$ has a unique extension to $\CH$ which we will denote by $\mu^\CH$. The $\CP^\CH$-q.s.\ order on $L^\infty_{\mbf c^\CH}$ is denoted by $\peq^\CH$.
The universal completion of the Borel $\sigma$-algebra plays an important role in discrete-time financial modeling under uncertainty as it admits the application of measurable selection arguments in the iterative construction of superhedging strategies \cite{BN} or for the optimal solution of the Merton problem \cite{Ba19}. Similarly in continuous time universal completions guarantee the concatenation property over the path space (e.g. \cite{BBKN14,NV}). Moreover, it is sometimes convenient to employ \textit{medial limits} (\cite{BCK19,Nutz}).

If $\CP$ is of class (S), we shall consider the \emph{supported enlargement} $(\CS,\CQ^\sharp)$, where $\CS:=\CF( \sca_+)$ is the completion along all supported measures, and $\CQ^\sharp:=\{\QW^\sharp\mid\QW\in\CQ\}$ is the set of extensions of a supported alternative $\CQ\approx\CP$ to $\CS$. The resulting upper probability on $\CS$ is denoted by $\mbf c^\sharp$, the $\CQ^\sharp$-q.s.\ order on the space $L^\infty_{\mbf c^\sharp}$ by $\peq^\sharp$. Note that $L^\infty_{\mbf c^\sharp}$ does not depend on the particular choice of the supported alternative $\CQ$, but any choice produces the same space. 

Completions along a particular set of probability measures are crucial in \cite[Section 7]{STZ} and throughout \cite{Cohen} to construct a conditional version of sublinear expectations.
They are adopted for the pathwise construction of the stochastic integral in \cite{NutzPathwise} and the resolution of the Holmstr\"om \& Milgrom contracting problem under uncertainty \cite{MP18}.
In \cite{BFM16} the underlying $\sigma$-algebra is completed along the set $\CP$ of all martingale measures for the underlying price process in order to obtain the measurability of the so-called ``arbitrage aggregator''.
We also refer to \cite{B+al16}, where the supported alternative of finitely supported martingale measures replaces the entire class of martingale measures considered in \cite{BFM16}.
In the statistical literature, this procedure is called \emph{completion of an experiment}; cf.\ \cite{Sufficiency,LeCam,Luschgy}.
We particularly highlight the discussion in \cite[Section 6]{Luschgy}.

\textit{A priori} it is clear that $\CH\subset\CS$ whenever $\CP$ is of class (S). 

\begin{lemma}\label{lem:propsuniversal}
Consider the enlargement $(\CH,\CP^\CH)$. For a signed measure $\mu\in ca$, $\mu\in{ca_{\mbf c}}$ is equivalent to $\mu^\CH\in ca_{\mbf c^\CH}$, and $\mu\in\sca$ is equivalent to $\mu^\CH\in sca_{\mbf c^\CH}$. In the latter case, $S(\mu)=S(\mu^\CH)$ holds.\end{lemma}
\begin{proof}
Let $\mu\in ca_+$ and suppose $N\in\CH$ satisfies $\mbf c^\CH(N)=0$. As $\CH\subset\CF(\{\mu\})$, we can choose $A\in\CF$, $A\subset N$, such that $N\backslash A\in\mf n(\mu)$. This entails $\mbf c(A)=\mbf c^\CH(A)=0$. Hence, we see that $\mu\ll\CP$ implies $\mu^\CH\ll\CP^\CH$.
The converse implication is clear.\\
Now suppose that $\mu\in\sca_+$ and let $S(\mu)\in\CF$ be its order support. Suppose $N\in\CH$ satisfies $N\subset S(\mu)$ and $\mu^\CH(N)=0$. Let $\P\in\CP$. As $\CH\subset\CF(\{\P\})$, there is $A\in\CF$, $A\subset N$, such that $N\backslash A\in\mf n(\P)$. From $\mu(A)=0$, we infer that $\mbf c(A)=0$ and that $\P^\CH(N)=\P(A)=0$. In conclusion, $S(\mu)$ is also the $\CP^\CH$-q.s.\ order support of $\mu^\CH$.\\
Conversely, note first that each finite measure $\nu$ on $(\Omega,\CH)$ satisfies $\nu=(\nu|_\CF)^\CH$. Hence, if $S(\mu^\CH)\in\CH$ is the $\CP^\CH$-q.s.\ order support of $\mu^\CH\in{sca_{\mbf c^\CH}}_+$, and $S\in\CF$ satisfies $S\subset S(\mu^\CH)$ and $S(\mu^\CH)\backslash S\in\mf n(\mu)$, it is straightforward to verify that $S$ is a version of the $\CP$-q.s.\ order support of $\mu$. 
\end{proof}

\begin{lemma}\label{prop:augmentation}
Suppose $\CP$ is of class \tn{(S)} and $\CQ$ is a supported alternative to $\CP$. Then $\CS=\CF(\CQ)$ and $J_\CS(\Linfty)$ is order dense and majorising in $L^\infty_{\mbf c^\sharp}$.
\end{lemma}
\begin{proof}
For the identity $\CS=\CF(\CQ)$, let $\CQ,\CR\approx\CP$ be supported alternatives to $\CP$ and assume that $\CR$ is disjoint. Then the experiment $(\Omega,\CF,\CQ)$ is majorised by the measure $\nu:=\sum_{\Q\in\CR}\Q$ which additionally satisfies $\nu\approx\CP$. By \cite[Lemma 10]{Luschgy}, 
\[\CF(\CQ)=\big\{A\subset\Omega\mid \forall\,F\subset\Omega:\,\nu(F)<\infty\, \Rightarrow~A\cap F\in\CF(\{\nu\})\big\}.\]
As the choice of $\CQ$ and $\CR$ was arbitrary, this shows that $\CS=\CF(\CQ)$. 

For the second assertion, $J_\CS(\Linfty)$ is majorising because the space contains equivalence classes of constant random variables. For its order density, it suffices to consider $B\in\CS$ with $\ind_B\neq 0$ in $L^\infty_{\mbf c^\sharp}$. Fix $\Q_*\in\CQ$ such that $\Q_*^\sharp(B)>0$. As $\CS\subset\CF(\{\Q_*\})$, there is $A\in\CF$ such that $A\subset B\cap S(\Q_*)$ and $\Q_*(A)=\Q_*^\sharp(B)>0$. 
Hence, $0\prec^\sharp \ind_A\peq^\sharp \ind_B$.
\end{proof}

The last necessary technical lemma is due to \cite[p.\ 45]{Fell}. {\bf CH} denotes the continuum hypothesis: there is no set $\mf X$ whose cardinality satisfies $|\N|=\aleph_0<|\mf X|<2^{\aleph_0}=|\R|$. 

\begin{lemma}\label{rem:continuum}
In {\bf ZFC+CH} let $\CP\subset\Delta(\CF)$ be of class (S) and suppose that there is a disjoint supported alternative to $\CP$ of cardinality $0<\kappa\le 2^{\aleph_0}$. Then $\ca^*=\Linfty$ if and only if $\CP$ is  Dedekind complete. In that case, we also have $\ca=\sca$.
\end{lemma}

We now exclude Dedekind completeness of the universal enlargement in many situations. 
The result is in the spirit of \cite{Rao} and applies, for instance, in the situation of Example~\ref{ex:big}.

\begin{corollary}\label{cor:Polish}
In {\bf ZFC+CH} assume  that $\CP$ is of class \tn{(S)}. 
\begin{itemize}
\item[\tn{(1)}]If there is a disjoint supported alternative of cardinality $0<\kappa\le 2^{\aleph_0}$ and if ${ca_{\mbf c}}\backslash\sca$ is nonempty, then $L^\infty_{\mbf c^\CH}$ is not Dedekind complete.
\item[\tn{(2)}]If $\Omega$ is Polish, $\CF$ is the Borel-$\sigma$-algebra on $\Omega$, $\CP$ is a set of Borel probability measures, and a disjoint supported alternative $\CQ$ contains a perfect or uncountable analytic or Borel subset $\CR$ of $\Delta(\CF)$, then $L^\infty_{\mbf c^\CH}$ is not Dedekind complete. 
\end{itemize}
\end{corollary}
\begin{proof}For (1), $\CP$ being of class (S) implies that $\CP^\CH$ is of class (S). Moreover, $\CP^\CH$ also admits a disjoint supported alternative of cardinality $\kappa\le 2^\aleph_0$ (Lemma~\ref{lem:propsuniversal}). Again by Lemma~\ref{lem:propsuniversal}, ${ca_{\mbf c}}\backslash\sca \neq\emptyset$ implies $ca_{\mbf c^\CH}\backslash sca_{\mbf c^\CH}\neq\emptyset$.
By Lemma~\ref{rem:continuum}, $L^\infty_{\mbf c^\CH}$ cannot be Dedekind complete. 

As for (2), we first observe that $\Delta(\CF)$ is Polish by \cite[Theorem 15.15]{Ali} and its cardinality does not exceed the cardinality of the continuum. Hence, any disjoint supported alternative $\CQ\approx\CP$ must satisfy $0<|\mf Q|\le 2^{\aleph_0}$. Apply Proposition~\ref{prop:Polish}(1) and use (1) above.
\end{proof}

The second main result complements \cite[Lemma 11 \& Theorem 6(b)]{Luschgy}. 

\begin{theorem}\label{thm:disjointsupp}
Suppose $\CP$ is of class \tn{(S)} with supported alternative $\CQ$, and that an enlargement $(\CG,\w\CP)$ completes $\Linfty$. Then the following assertions hold:
\begin{itemize}
\item[\tn{(1)}]Each $\mu\in\sca$ extends uniquely to a $\w\mu\in sca_{\mbf{\w c}}$.
\item[\tn{(2)}]Each supported alternative $\CQ\approx\CP$ extends to a supported alternative $\w\CQ:=\{\w\QW\mid \QW\in \CQ\}\approx\w\CP$.
\item[\tn{(3)}]One may assume $\CG=\CG\big(\w\CQ\big)$. More precisely, if in the situation of \tn{(2)} we consider the enlargement $(\CG(\w\CQ),\w\CQ^\sharp)$, $J_{\CG(\w\CQ)}$ is a lattice isomorphism. 
\end{itemize}
\end{theorem}
\begin{proof} To prove statement (1), note that $(\Linfty)^\sim_n$ and $(L^\infty_{\mbf{\w c}})^\sim_n$ are by \cite[Theorem 1.84]{AliBurk} and its proof lattice isomorphic via the bijection given on positive functionals $\phi\in{(\Linfty)^\sim_n}$ as in \eqref{eq:ext:2}.
By Proposition~\ref{prop:(A1)-(A3)}(3), $(\Linfty)^\sim_n=\sca$ and $(L^\infty_{\mbf{\w c}})^\sim_n=sca_{\mbf{\w c}}$.

\smallskip

For (2), the set $\w\CQ:=\{\w\Q\mid \Q\in\CQ\}$ associated with a supported alternative $\CQ\approx\CP$ is a subset of $sca_{\mbf{\w c}}$; cf.\ (1). In order to see that $\w\CQ\approx\w\CP$, use order density of $J_\CG(\Linfty)$ in $L^\infty_{\mbf{\w c}}$ to verify for each $B\in\CG$ that 
\[\sup_{\Q\in\CQ}\w\Q(B)=\sup\{\phi_\Q(X)\mid \Q\in\CQ,\,X\in\Linfty,\,J_\CG(X)\peq_\CG\ind_B\}.\]
Here, $\peq^\CG$ denotes the $\w\CP$-q.s.\ order on $L^\infty_{\w{\mbf c}}$. 
The right-hand side is positive if and only if $\ind_B\neq 0$ in $L^\infty_{\mbf{\w c}}$.

\smallskip

At last, in order to verify (3), one considers the enlargement $(\CG(\w\CQ),\w\CQ^\sharp)$ of $(\CG,\w\CP)$ and uses Lemma~\ref{prop:augmentation} to show that $\CY:=J_{\CG(\w\CQ)}(L^\infty_{\mbf{\w c}})$ is an order dense and majorising Dedekind complete sublattice of the enlarged space. By \cite[Theorem 1.40]{AliBurk}, $\CY$ is a majorising ideal, i.e.\ it has to agree with the enlarged space.
\end{proof}

\begin{remark}
Theorem~\ref{thm:disjointsupp}(3) means that {\em if} Dedekind completion of $\Linfty$ is obtained by an enlargement $(\CG,\w\CP)$ and {\em if} $\CP$ is of class (S), then $\CG$ is {\em at least as large} as the supported completion $\CS$ of $\CF$. 
Morally speaking, this suggests together with  Corollary~\ref{cor:Polish} that one should not expect $L^\infty_{\mbf c^\CH}$ to be Dedekind complete since $\CH\subsetneq\mathcal S$ is possible. The universal enlargement is therefore not a sensible choice if the aim is aggregation. Note that the class (S) assumption does not pose a severe restriction here; class (S) cannot be disproved under Dedekind completeness (Corollary~\ref{cor:example}(2)).
\end{remark}

\section{Proof of Proposition \ref{rob:bin}}\label{proof:bin}
$\CP \approx \CQ$ holds by construction, we therefore only need to show that every $\QW\in\CQ$ is supported. From the construction in Section~\ref{sec:BN}, there are stochastic kernels  $Q_t(\omega;\cdot)\in \CQ_{t}(\omega)$ for any $\omega\in \Omega_{t-1}$ and $Q_t(\omega;\cdot)\circ Y_{t}^{-1}\in \mathcal{LW}_{t}(\omega)$ such that 
$\QW=Q_1\otimes \ldots \otimes Q_{T}$.
Let $(O_n)_{n\in\N}$ be a countable basis for the topology on $(0,\infty)$ and $C_n=O_n^c$. From \cite[Lemma 4.3]{BN} we observe that the set-valued map
\[\Lambda\colon\begin{array}{l}\Omega_{t-1}\to2^{(0,\infty)},\\
\omega\mapsto\bigcap\big\{C_n\mid n\in\N,\,Q_t(\omega; Y_{t} \in C_n)=1\big\},\end{array}\]
is universally measurable and filters out the topological support of the distribution of $Y_t$ under $Q_t(\om,\cdot)$. Fix an $\omega\in\Omega_{t-1}$: By construction $Q_t(\omega;\cdot)\circ Y_{t}^{-1}\in \mathcal{L}_{t}(\omega)$, which implies $Q_t(\omega;\cdot)\circ Y_{t}^{-1}= \pi(\omega)\delta_{a(\omega)}+(1-\pi(\omega))\delta_{b(\omega)}$ 
for some $a(\omega)\in [u_{t-1}(\omega),U_{t-1}(\omega)]$ and $b(\omega)\in [d_{t-1}(\omega),D_{t-1}(\omega)]$. In particular,
$$\Lambda(\omega)=\{a(\omega),b(\omega)\}=\{m_{t-1}(\omega),M_{t-1}(\omega)\},$$
where the maps $m_{t-1}(\omega)=\min \Lambda(\omega)$, $M_{t-1}(\omega)=\max\Lambda(\omega)$ are $\CF_{t-1}$-measurable (with $\CF_{t-1}$ being the universal completion of the Borel $\sigma$-algebra $\mathcal B(\Omega_t)$). We observe that $Q_t(\omega;\cdot)\circ Y_{t}^{-1}= \tilde{\pi}(\omega)\delta_{m_{t-1}(\omega)}+(1-\tilde{\pi}(\omega))\delta_{M_{t-1}(\omega)}$.  Therefore the set $S(\QW)$ defined by
\[\{(x_0,\ldots,x_T)\in \Omega\mid \forall\,t\in \{0,\ldots,T-1\}:~Y_{t+1}(x_{t+1})=m_{t}(x_0,\ldots,x_{t})\tn{ or }Y_{t+1}(x_{t+1})=M_{t}(x_0,\ldots,x_{t})\} \]
is the candidate to be the order support of $\QW$. Indeed, setting $W_d^1(x_0,\ldots,x_T)=Y_{1}(x_{1})-m_{0}$, $W_u^1(x_0,\ldots,x_T)=Y_{1}(x_{1})-M_{0}$ and  
\[\begin{array}{l}
    W_d^t(x_0,\ldots,x_T)= Y_{t}(x_{t})-m_{t}(x_0,\ldots,x_{t-1}),\\
    W_u^t(x_1,\ldots,x_T)= Y_{t}(x_{t})-M_{t}(x_0,\ldots,x_{t-1}),
\end{array}\quad\quad(x_0,\ldots,x_T)\in\Omega,~t=2,\ldots,T.\]
all $W^t_d,W^t_u$ are $\CF$-measurable random variables, and %to $\R$ 
\[S(\Q):=\bigcap_{t=1}^T (W_d^t)^{-1}(\{0\})\cup(W_u^t)^{-1}(\{0\})\in\CF.\] 
To show that $\QW(S(\QW))=1$ we observe that for any $\omega\in\Omega_{t-1}$ we have by construction $Q_t(\omega;A_{t})=1$ where
$A_{t}=(W_d^t)^{-1}(\{0\})\cup (W_u^t)^{-1}(\{0\})$. 
Therefore
\[
Q_t(x_1,\ldots,x_{t-1}; A_{t}) =\int_{0}^{\infty}\ind_{A_{t}}(x_1,\ldots,x_T) Q_{t}(x_1,\ldots,x_{t-1}; dx_{t}) = 1,\]
which implies 
\[\QW(A_{t})=\int_{0}^{\infty}\dots \int_{0}^{\infty}\ind_{A_{t}}(x_1,\ldots,x_T) Q_{t}(x_1,\ldots,x_{t-1}; dx_{t})\dots  Q_1(dx_1)= 1.\]
Finally we notice that for any $\QW,\QW'\in \CQ$ either $S(\QW)\equiv S(\QW')$ or $S(\QW)\cap S(\QW')=\emptyset$, as every $Y_t$ is a bijection from $(0,\infty)$ to $(0,\infty)$.

\end{document}